\documentclass[10pt, a4paper, reqno, oneside]{amsart}

\usepackage{amsmath, amsopn, amsfonts, amsthm, amssymb, amscd, enumerate, multicol, scalefnt}
\usepackage{sansmath}
\usepackage{etex}
\usepackage{hyperref}
\usepackage{t1enc}
\usepackage{relsize}
\usepackage[latin1]{inputenc}
\usepackage{graphicx}
\usepackage[all]{xy}
\usepackage{color}
\usepackage{slashed}
\usepackage[mathscr]{euscript}
\usepackage{enumitem}
\setlist[itemize]{noitemsep, topsep=1pt, leftmargin=20pt}
\newcommand\bcdot{\ensuremath{
  \mathchoice
   {\mskip\thinmuskip\lower0.2ex\hbox{\scalebox{1.6}{$\cdot$}}\mskip\thinmuskip}}
   {\mskip\thinmuskip\lower0.2ex\hbox{\scalebox{1.6}{$\cdot$}}\mskip\thinmuskip}
   {\lower0.3ex\hbox{\scalebox{1.2}{$\cdot$}}}
   {\lower0.3ex\hbox{\scalebox{1.2}{$\cdot$}}}
}
\theoremstyle{plain}
\newtheorem{theo}{Theorem}[section]

\newtheorem{prop}[theo]{Proposition}

\theoremstyle{definition}
\newtheorem{rem}[theo]{Remark}
\newtheorem{example}[theo]{Example}
\newtheorem{definition}[theo]{Definition}

\theoremstyle{plain}
\newtheorem{lemma}[theo]{Lemma}
\newtheorem{theorem}[theo]{Theorem}
\newtheorem{corollary}[theo]{Corollary}

\theoremstyle{definition}

\newtheorem{remark}[theo]{Remark}

\theoremstyle{plain}
\newtheorem{thmint}{Theorem}

\newtheorem{corint}[thmint]{Corollary}

\renewcommand{\=}{:=}

\renewcommand{\a}{\alpha}
\renewcommand{\b}{\beta}
\renewcommand{\c}{\chi}
\renewcommand{\d}{\delta}
\newcommand{\e}{\varepsilon}
\newcommand{\f}{\varphi}
\newcommand{\g}{\gamma}

\renewcommand{\l}{\lambda}

\renewcommand{\L}{\Lambda}

\newcommand{\W}{\Omega}

\newcommand{\bC}{\mathbb{C}}

\newcommand{\bR}{\mathbb{R}}
\newcommand{\bZ}{\mathbb{Z}}

\newcommand{\bN}{\mathbb{N}}

\newcommand{\bQ}{\mathbb{Q}}

\newcommand{\fG}{\mathsf{G}}
\newcommand{\fH}{\mathsf{H}}

\newcommand{\fL}{\mathsf{L}}

\newcommand{\fT}{\mathsf{T}}
\newcommand{\fGL}{\mathsf{GL}}
\newcommand{\fSL}{\mathsf{SL}}

\newcommand{\fO}{\mathsf{O}}
\newcommand{\fSO}{\mathsf{SO}}
\newcommand{\fSE}{\mathsf{SE}}

\newcommand{\fSU}{\mathsf{SU}}

\newcommand{\ga}{\mathfrak{a}}

\renewcommand{\gg}{\mathfrak{g}}
\newcommand{\gh}{\mathfrak{h}}
\newcommand{\gi}{\mathfrak{i}}

\newcommand{\gm}{\mathfrak{m}}

\newcommand{\go}{\mathfrak{o}}

\newcommand{\gs}{\mathfrak{s}}
\newcommand{\gt}{\mathfrak{t}}

\newcommand{\gz}{\mathfrak{z}}

\newcommand{\gS}{\mathfrak{S}}

\newcommand{\so}{\mathfrak{so}}
\newcommand{\su}{\mathfrak{su}}
\newcommand{\cC}{\mathcal{C}}

\newcommand{\cR}{\mathcal{R}}

\newcommand{\cV}{\mathcal{V}}

\newcommand{\eB}{\EuScript{B}}

\newcommand{\eH}{\EuScript{H}}

\newcommand{\st}{{\operatorname{st}}}

\newcommand{\p}{\partial}

\newcommand{\rar}{\rightarrow}

\newcommand{\la}{\langle}
\newcommand{\ra}{\rangle}

\renewcommand{\square}{\kern1pt\vbox
{\hrule height 0.6pt\hbox{\vrule width 0.6pt\hskip 3pt \vbox{\vskip
6pt}\hskip 3pt\vrule width 0.6pt}\hrule height0.6pt}\kern1pt}
\renewcommand{\=}{\  \raisebox{0.15mm}{:} {=} \ }
\newcommand{\rank}{\operatorname{rank}}

\DeclareMathOperator\Tr{Tr}
\DeclareMathOperator\Lie{Lie}

\DeclareMathOperator\ad{ad}

\DeclareMathOperator{\vspan}{span}

\newcommand\Rm{\operatorname{Rm}}

\newcommand{\inj}{\operatorname{inj}}

\newcommand{\wt}{\widetilde}

\newcommand{\ol}{\overline}

\newcommand{\zero}{\operatorname{o}}
\def\<#1,#2>{\langle\,#1,\,#2\,\rangle}

\newcommand{\Aac}{\`A}

\newcommand{\Math}{{\it Mathematica\raise5 pt\hbox{$\scriptscriptstyle \circledR$}7}}
\newcommand{\n}{\nabla}

\newcommand{\beq}{\begin{equation}}
\newcommand{\eeq}{\end{equation}}
\def\<#1,#2>{\langle\,#1,\,#2\,\rangle}
\newcommand{\arr}{\begin{array}{rlll}}
\newcommand{\ea}{\end{array}}
\newcommand{\bea}{\begin{eqnarray}}
\newcommand{\eea}{\end{eqnarray}}
\newcommand{\bean}{\begin{eqnarray*}}
\newcommand{\eean}{\end{eqnarray*}}

\def\sideremark#1{\ifvmode\leavevmode\fi\vadjust{%            The remark
\vbox to0pt{\hbox to 0pt{\hskip\hsize\hskip1em%               will appear only
\vbox{\hsize3cm\tiny\raggedright\pretolerance10000%          on the side
\noindent #1\hfill}\hss}\vbox to8pt{\vfil}\vss}}}%           in 3cm
\newcounter{ssig}
\newcounter{ttig}
\title[Convergence of locally homogeneous spaces]{Convergence of locally homogeneous spaces}
\author{Francesco Pediconi}

\thanks{This work was supported by project PRIN 2017 ``Real and Complex Manifolds: Topology, Geometry and holomorphic dynamics'' (code 2017JZ2SW5) and by GNSAGA of INdAM} 

\begin{document}
\begin{abstract} We study three different topologies on the moduli space $\eH^{\rm loc}_m$ of equivariant local isometry classes of $m$-dimensional locally homogeneous Riemannian spaces. As an application, we provide the first examples of locally homogeneous spaces converging to a limit space in the pointed $\cC^{k,\a}$-topology, for some $k>1$, which do not admit any convergent subsequence in the pointed $\cC^{k+1}$-topology.
\end{abstract}

\maketitle

%\tableofcontents

\section{Introduction} \setcounter{equation} 0

In this paper we consider $m$-dimensional locally homogeneous Riemannian spaces and the moduli space of their equivalence classes up to equivariant local isometries, which we denote by $\eH^{\rm loc}_m$. We also denote by $\eH_m \subset \eH^{\rm loc}_m$ the moduli subspace of locally homogeneous spaces which are equivariantly locally isometric to a globally homogeneous space. It is known that $\eH_m = \eH^{\rm loc}_m$ if $1 \leq m \leq 4$ (see \cite[Sec 7]{Mos}), while $\eH_m \subsetneq \eH^{\rm loc}_m$ for any $m \geq 5$ (see \cite{Mos,Kow}). Moreover, $\eH^{\rm loc}_m$ is a natural completion of $\eH_m$ in case $m \geq 5$ (see e.g. \cite{Ped3}). \smallskip

In \cite{Lau1} Lauret developed a framework which allows one to parameterize the moduli space $\eH_m$ by a distinguished set of Lie algebras with an additional structure. This parametrization associates to any element $\mu \in \eH_m$ a pair $(\fG_{\mu}/\fH_{\mu},g_{\mu})$ given by the quotient of a simply connected Lie group $\fG_{\mu}$ by a closed connected subgroup $\fH_{\mu} \subset \fG_{\mu}$ and a $\fG_{\mu}$-invariant metric $g_{\mu}$. Furthermore, he endowed $\eH_m$ with three different topologies and he discussed the relations among them. These are: \begin{itemize}[leftmargin=20pt]
\item[$\bcdot$] the {\it pointed convergence topology}, that is the usual convergence in pointed Cheeger-Gromov topology of pointed Riemannian manifolds;
\item[$\bcdot$] the {\it infinitesimal convergence topology}, that is a weaker notion that involves only the germs of the metrics at a point;
\item[$\bcdot$] the {\it algebraic convergence topology}, which only takes into account the underlying algebraic structure.
\end{itemize} Since this approach turned out to be particularly well suited for studying curvature variational problems and geometric flows in the globally homogeneous setting (see e.g. \cite{Lau1,Lau2,Lau3,Lau4} and references therein), we carry out in this paper a similar analysis for $\eH^{\rm loc}_m$. \smallskip

Some classical results in \cite{Sp1,Sp2} allows us to reformulate Lauret's construction for $\eH^{\rm loc}_m$ (see also \cite[Sec 5]{B\"o4}). More precisely, one can associate to any $\mu \in \eH^{\rm loc}_m$ a pair $(\fG_{\mu}/\fH_{\mu},g_{\mu})$ as above, with the only difference that, if $\mu \in \eH^{\rm loc}_m \setminus \eH_m$, then the subgroup $\fH_{\mu} \subset \fG_{\mu}$ is not closed. Notice that in this case the quotient space $\fG_{\mu}/\fH_{\mu}$ is not even Hausdorff, and hence it is not a smooth manifold. This issue is overcome by considering {\it local factor spaces}, which generalize the usual quotients of Lie groups (see \cite{Mos,Sp2}).

Both the infinitesimal convergence and the algebraic convergence perfectly extend to the space $\eH^{\rm loc}_m$ (see Definition \ref{algconv} and Definition \ref{infconv}). Actually, for what concerns the former, we introduce a weaker version of that, which we call {\it $s$-infinitesimal convergence}. Here, in the following definition, we denote by $\imath(m)$ the maximum of the Singer invariants of $m$-dimensional locally homogeneous spaces (see Formula \eqref{i(m)}).

\vskip 5pt

\noindent \textbf{Definition (see Definition \ref{infconv}). } For any integer $s \geq \imath(m)+2$, a sequence $(\mu^{(n)}) \subset \eH^{\rm loc}_m$ {\it converges $s$-infinitesimally} to $\mu^{(\infty)} \in \eH^{\rm loc}_m$ if the Riemannian curvature tensors and their first $s$ covariant derivatives at the origin $e_{\mu^{(n)}}\fH_{\mu^{(n)}}$ of $(\fG_{\mu^{(n)}}/\fH_{\mu^{(n)}},g_{\mu^{(n)}})$ converge to those of $(\fG_{\mu^{(\infty)}}/\fH_{\mu^{(\infty)}},g_{\mu^{(\infty)}})$ at the origin $e_{\mu^{(\infty)}}\fH_{\mu^{(\infty)}}$.

\vskip 5pt

This is motivated by the fact that any class $\mu \in \eH^{\rm loc}_m$ is completely determined by the curvature and its covariant derivatives at some point up to order $\imath(m)+2$ (see \cite{Si}). Note that Lauret's original definition is equivalent to ours in case $s=\infty$ and, for the convenience of the reader, we will prove this in Appendix \ref{appendixA}. \smallskip

However, there are inconveniences which arise in trying to extend the pointed convergence topology to $\eH^{\rm loc}_m$. This is due to the fact that a rigorous definition of local factor spaces depends on arbitrary choices of distinguished neighborhoods inside $\fH_{\mu}$ and $\fG_{\mu}$ (see e.g. \cite[Sec 6]{Ped2}), and hence this construction seems to be not appropriate for studying pointed convergence. \smallskip

In \cite[Thm A]{Ped3}, we proved that any locally homogeneous space $(\fG_{\mu}/\fH_{\mu},g_{\mu})$ with $|\sec(g_{\mu})|\leq1$ is locally isometric to a unique, up to equivariant isometry, Riemannian distance ball $(\eB_{\mu},\hat{g}_{\mu})=(\eB_{\hat{g}_{\mu}}(o_{\mu},\pi),\hat{g}_{\mu})$ of radius $\pi$ with $\inj_{o_{\mu}}(\eB_{\mu},\hat{g}_{\mu})=\pi$, which we call {\it geometric model} (see Theorem \ref{great}). These objects provide a useful parametrization for the subspace $\eH^{\rm loc}_m(1) \subset \eH^{\rm loc}_m$ of those equivalence classes $\mu \in \eH^{\rm loc}_m$ with bounded sectional curvature $|\sec(g_{\mu})|\leq1$, and can be used to study pointed convergence of locally homogeneous spaces (see e.g. \cite[Thm B]{Ped3}). Let us stress that, by the very definition, it follows that the pointed convergence in the $\cC^{s+2}$-topology of a sequence of geometric models clearly implies the $s$-infinitesimal convergence. Our first result is a kind of converse of such a statement. Namely

%\begin{thmint} If a sequence $(\mu^{(n)}) \subset \eH^{\rm loc}_m$ converges $(s{+}1)$-infinitesimally to $\mu^{(\infty)}\in \eH^{\rm loc}_m$ for some integer $s \geq \imath(m)+2$, then the corresponding geometric models $(B^m,\hat{g}_{\mu^{(n)}})$ converge locally to the geometric model $(B^m,\hat{g}_{\mu^{(\infty)}})$ in the pointed $\cC^{s+2,\a}$-topology for any $0\leq\a<1$. \label{MAIN-D} \end{thmint}

\begin{thmint} If a sequence $(\mu^{(n)}) \subset \eH^{\rm loc}_m(1)$ converges $(s{+}1)$-infinitesimally to $\mu^{(\infty)}\in \eH^{\rm loc}_m(1)$ for some integer $s \geq \imath(m)+2$, then the corresponding geometric models $(\eB_{\mu^{(n)}},\hat{g}_{\mu^{(n)}})$ converge to the geometric model $(\eB_{\mu^{(\infty)}},\hat{g}_{\mu^{(\infty)}})$ in the pointed $\cC^{s+2,\a}$-topology for any $0\leq\a<1$. \label{MAIN-A} \end{thmint}

%\begin{thmint} Let $(\mu^{(n)}) \subset \eH^{\rm loc}_m(1)$ be a sequence, $\mu^{(\infty)} \in \eH^{\rm loc}_m(1)$ and $s \geq \imath(m)+2$ an integer. \begin{itemize}
%\item[i)] If $(\eB_{\mu^{(n)}},\hat{g}_{\mu^{(n)}})$ converges to $(\eB_{\mu^{(\infty)}},\hat{g}_{\mu^{(\infty)}})$ in the pointed $\cC^{s+2}$-topology, then $(\mu^{(n)})$ converges $s$-infini\-tesimally to $\mu^{(\infty)}$.
%\item[ii)] If $(\mu^{(n)})$ converges $(s{+}1)$-infinitesimally to $\mu^{(\infty)}$, then $(\eB_{\mu^{(n)}},\hat{g}_{\mu^{(n)}})$ converges to $(\eB_{\mu^{(\infty)}},\hat{g}_{\mu^{(\infty)}})$ in the pointed $\cC^{s+2,\a}$-topology for any $0<\a<1$.
%\end{itemize} \label{MAIN-A} \end{thmint}

Actually, we do not know whether this statement is optimal, i.e. we do not know whether the $s$-infinitesimal convergence implies the pointed convergence of geometric models in the $\cC^{s+2}$-topology or not. However, Theorem \ref{MAIN-A} immediately implies the following

\begin{corint} A sequence $(\mu^{(n)}) \subset \eH^{\rm loc}_m(1)$ converges infinitesimally to $\mu^{(\infty)}\in \eH^{\rm loc}_m(1)$ if and only if the corresponding geometric models $(\eB_{\mu^{(n)}},\hat{g}_{\mu^{(n)}})$ converge to the geometric model $(\eB_{\mu^{(\infty)}},\hat{g}_{\mu^{(\infty)}})$ in the pointed $\cC^{\infty}$-topology. \label{MAINCOR-B} \end{corint}

%\begin{corint} Let $(\mu^{(n)}) \subset \eH^{\rm loc}_m(1)$ be a sequence and $\mu^{(\infty)} \in \eH^{\rm loc}_m(1)$. \begin{itemize}
%\item[i)] The geometric models $(\eB_{\mu^{(n)}},\hat{g}_{\mu^{(n)}})$ converge to $(\eB_{\mu^{(\infty)}},\hat{g}_{\mu^{(\infty)}})$ in the pointed $\cC^{\infty}$-topology if and only if $(\mu^{(n)})$ converges infinitesimally to $\mu^{(\infty)}$.
%\item[ii)] If $(\mu^{(n)})$ converges algebraically to $\mu^{(\infty)}$, then $(\eB_{\mu^{(n)}},\hat{g}_{\mu^{(n)}})$ converge to $(\eB_{\mu^{(\infty)}},\hat{g}_{\mu^{(\infty)}})$ in the pointed $\cC^{\infty}$-topology.
%\end{itemize} \label{MAINCOR-B} \end{corint}

Let us remark that the pointed convergence is much stronger than the infinitesimal convergence (see \cite[Subsec 6.2]{Lau1}, \cite[Subsec 3.4]{Lau4}). On the contrary, Corollary \ref{MAINCOR-B} shows that the pointed convergence of geometric models is indeed equivalent to the infinitesimal convergence. This special feature is due to the fact that, by the very definition, collapse cannot occur along a sequence of geometric models. \smallskip

Furthermore, we investigate $s$-infinitesimal convergence by allowing $s$ to vary. The main result of this paper shows that keeping all the covariant derivatives of the curvature tensor bounded along a sequence of locally homogeneous spaces is a much more restrictive condition than just bounding a finite number of them. More precisely

\begin{thmint} For any choice of $m,s \in \bN$ such that $m\geq3$ and $s \geq \imath(m)+2$, the notion of $s$-infinitesimal convergence in $\eH^{\rm loc}_m$ is strictly weaker than that of $(s{+}1)$-infinitesimal convergence. \label{MAIN-C} \end{thmint}

In order to prove Theorem \ref{MAIN-C}, we construct an explicit 2-parameter family $$\{\mu_{\star}(\e,\d): {\e, \d \in \bR} \, ,  \,\, {\e>0} \, , \,\, {0\leq \d <1} \} \subset \eH_3$$ with the following property: for any fixed integer $k\geq0$, there are $(\e^{(n)}), (\d^{(n)}) \subset (0,1)$ and $C>0$ such that, letting $\mu^{(n)} \= \mu_{\star}(\e^{(n)},\d^{(n)})$, it holds that $$\begin{gathered}
\sum_{i=0}^k\big|({\n}^{g_{\mu^{(n)}}})^i\Rm(g_{\mu^{(n)}})\big|_{g_{\mu^{(n)}}} \leq C \quad \text{ for any $n \in \bN$ } \\
\text{ and } \quad \big|({\n}^{g_{\mu^{(n)}}})^{k+1}\Rm(g_{\mu^{(n)}})\big|_{g_{\mu^{(n)}}} \rar +\infty \quad \text{ as $n\rar+\infty$ } \,\, .
\end{gathered}$$ The key idea to obtain such families $\mu_{\star}(\e,\d)$ is to consider a slight modification of the Berger spheres, which arise from the canonical variation of the round metric on $\fSU(2)$ with respect to the Hopf fibration $S^1 \rar \fSU(2) \rar \bC P^1$ (see \cite[p. 252]{Bes}). \smallskip

Combining Theorem \ref{MAIN-A} and Theorem \ref{MAIN-C}, we also obtain

\begin{corint} For any $m,k \in \bN$ with $m\geq3$ and $k\geq \imath(m)+4$, there exists a sequence of $m$-dimensional locally homogeneous spaces converging to a limit locally homogeneous space in the pointed $\cC^{k,\a}$-topology for any $0<\a<1$ and which does not admit a convergent subsequence in the pointed $\cC^{k+1}$-topology. \label{MAINCOR-D} \end{corint}

To the best of our knowledge, Corollary \ref{MAINCOR-D} provides the first example of locally homogeneous spaces converging to a (smooth) locally homogeneous space in the pointed $\cC^{k,\a}$-topology, for some fixed integer $k>1$, which do not admit any convergent subsequence in the pointed $\cC^{k+1}$-topology. \smallskip

Finally, we stress that the main results in \cite{Ped3} imply that {\it the moduli space $\eH^{\rm loc}_m(1)$ is compact in the pointed $\cC^{1,\a}$-topology for any $0<\a<1$, for any integer $m\geq1$}. The contents of this paper lead to the following refinement.

\begin{corint} If $1\leq m<3$, then $\eH^{\rm loc}_m(1)=\eH_m(1)$ is compact in the pointed $\cC^{\infty}$-topology. On the contrary, $\eH^{\rm loc}_m(1)$ is not compact in the pointed $\cC^3$-topology for any $m\geq3$. \label{MAINCOR-E} \end{corint}

The paper is structured as follows. In Section \ref{prel} we collect some preliminaries on locally homogeneous spaces. In Section \ref{algasp} we introduce the notion of $s$-infinitesimal convergence and we prove Theorem \ref{MAIN-A}. In Section \ref{AB} we introduce the notion of {\it almost-Berger sequence on $\fSU(2)$} and we provide some estimates for the curvature tensor and its covariant derivatives along these sequences. In Section \ref{mainsec} we prove Theorem \ref{MAIN-C}, Corollary \ref{MAINCOR-D} and Corollary \ref{MAINCOR-E}. Finally, in Appendix \ref{appendixA} we give a proof of the following facts: (i) the subspace $\eH_m$ is dense in $\eH^{\rm loc}_m$ with respect to the algebraic convergence topology, and (ii) our notion of infinitesimal convergence is equivalent to the one introduced by Lauret in \cite{Lau1}. \medskip

\noindent{\it Acknowledgement.} We are grateful to Christoph B\"ohm for his teachings and suggestions. We warmly thank Ramiro Lafuente, Andrea Spiro and Luigi Verdiani for helpful remarks. Finally, we would like to thank the anonymous referees for their careful reading of the manuscript and useful comments.

\section{Preliminaries on locally homogeneous Riemannian spaces} \label{prel}

\subsection{Locally homogeneous Riemannian manifolds} \label{Nomizualg} \hfill \par

A smooth Riemannian manifold $(M,g)$ is {\it locally homogeneous} if the pseudogroup of local isometries of $(M,g)$ acts transitively on $M$, i.e. if for any $x,y \in M$ there exist two open sets $U_x , U_y \subset M$ and a local isometry $f : U_x \rar U_y$ such that $x \in U_x$, $y \in U_y$ and $f(x) = y$. It is known that any such a space admits a unique real analytic structure, up to real analytic diffeomorphism, with respect to which the metric $g$ is real analytic (see e.g. \cite[Thm 2.2]{Sp1} or \cite[Lemma 1.1]{BLS}). \smallskip

Given a locally homogeneous Riemannian space $(M,g)$ and a distinguished point $p \in M$, the {\it Killing generators at $p$} are defined as those pairs $(v,A) \in T_pM \oplus \so(T_pM,g_p)$ such that \beq A \cdot g_p = 0 \,\, , \quad v \,\lrcorner\, \big((\n^g)^{k+1}\Rm(g)_p\big) + A \cdot \big((\n^g)^k\Rm(g)_p\big) = 0 \quad \text{ for any $k \in \bZ_{\geq0}$} \, , \label{killgen} \eeq where $\n^g$ denotes the Levi-Civita covariant derivative of $(M,g)$, $\Rm(g)(X\wedge Y) \= \n^g_{[X,Y]}-[\n^g_X,\n^g_Y]$ its Riemannian curvature operator and $\so(T_pM,g_p)$ acts on the tensor algebra of $T_pM$ as a derivation. This nomenclature is due to the following fact. For any Killing vector field $X$ of $(M,g)$ defined in a neighborhood of $p$, the pair $\big(X|_p,-(\n^gX)_p\big)$ is a Killing generator of $(M,g)$ at $p$. Conversely, being $(M,g)$ real analytic, by \cite[Thm 2]{No} there exists a neighborhood $\W_p \subset M$ of $p$ such that for any Killing generator $(v,A)$ at $p$, there exists a Killing vector field $X$ on $\W_p$ with $X|_p=v$ and $-(\n^gX)_p=A$. We denote by $\mathfrak{kill}^g$ the Lie algebra of all the Killing generators at $p$ with the Lie bracket \beq \big[(v,A),(w,B)\big] \= \big(A(w)-B(v),[A,B]+\Rm(g)_p(v,w)\big) \eeq and we call it the {\it Nomizu algebra of $(M,g,p)$}.

%We indicate by $\mathfrak{kill}^g$ the space of Killing generators of $(M,g)$ at $p$. This nomenclature is due to the following fact. For any Killing vector field $X$ of $(M,g)$ defined in a neighborhood of $p$, the pair $\big(X|_p,-(\n^gX)_p\big)$ is a Killing generator of $(M,g)$ at $p$. Conversely, being $(M,g)$ real analytic, by \cite[Thm 2]{No} there exists a neighborhood $\W_p \subset M$ of $p$ such that for any $(v,A) \in \mathfrak{kill}^g$ there exists a Killing vector field $X$ on $\W_p$ with $X|_p=v$ and $-(\n^gX)_p=A$. The vector space $\mathfrak{kill}^g$ is a Lie algebra with the Lie bracket \beq \big[(v,A),(w,B)\big] \= \big(A(w)-B(v),[A,B]+\Rm(g)_p(v,w)\big) \eeq and it is called the {\it Nomizu algebra of $(M,g,p)$}.

\subsection{Orthogonal transitive Lie algebras} \label{ssinfmod} \hfill \par

Let us firstly introduce the following notion, which will play an important role.

\begin{definition}[\cite{OV}, p. 51] Let $\fG$ be a connected Lie group, $\gh \subset \gg \= \Lie(\fG)$ a Lie subalgebra and $\fH \subset \fG$ the connected Lie subgroup with $\Lie(\fH)=\gh$. The {\it Malcev-closure of $\gh$ in $\fG$} is the Lie algebra $\ol{\gh}^{\mathsmaller{\fG}} \subset \gg$ of the closure $\ol{\fH}$ of $\fH$ in $\fG$. The Lie algebra $\gh$ is said to be {\it Malcev-closed in $\fG$} if $\gh=\ol{\gh}^{\mathsmaller{\fG}}$. \label{Malcl} \end{definition}

Notice that the notion of Malcev-closure depends on the Lie group $\fG$, i.e. it is possible that $\ol{\gh}^{\mathsmaller{\fG}} \neq \ol{\gh}^{\mathsmaller{\fG'}}$ for two connected Lie groups with $\Lie(\fG)=\Lie(\fG')=\gg$.

\begin{example} Let $\gg\=\bR^2$ and $\fG\=\bR^2$, $\fG'\=\fT^2=\bZ^2\backslash \bR^2$. We consider the Lie subalgebra $\gh \subset \gg$ spanned by the vector $(1,\sqrt{2})^t \in \bR^2$. Then it follows that $\gh=\ol{\gh}^{\mathsmaller{\fG}}\subsetneq \ol{\gh}^{\mathsmaller{\fG'}}=\gg$ (see Lemma \ref{dimLv}). \end{example}

Then, we recall the following

\begin{definition} Let $m,q \in \bZ_{\geq0}$. An {\it orthogonal transitive Lie algebra $(\gg=\gh+\gm,\la\,,\ra)$ of rank $(m,q)$} is the datum of \begin{itemize}
\item[$\bcdot$] a $(q+m)$-dimensional Lie algebra $\gg$;
\item[$\bcdot$] a $q$-dimensional Lie subalgebra $\gh \subset \gg$ which does not contain any non-trivial ideal of $\gg$;
\item[$\bcdot$] an $\ad(\gh)$-invariant complement $\gm$ of $\gh$ in $\gg$;
\item[$\bcdot$] an $\ad(\gh)$-invariant Euclidean product $\la\,,\ra$ on $\gm$.
\end{itemize} An orthogonal transitive Lie algebra $(\gg=\gh+\gm,\la\,,\ra)$ is said to be {\it regular} if $\gh$ is Malcev-closed in the simply connected Lie group $\fG$ with $\Lie(\fG)=\gg$, {\it non-regular} otherwise. \label{otLa} \end{definition}

Let $(\gg=\gh+\gm,\la\,,\ra)$ be an orthogonal transitive Lie algebra of rank $(m,q)$. Since there are no ideals of $\gg$ in $\gh$, the adjoint action of $\gh$ on $\gm$ is a faithful representation in $\so(\gm,\la\,,\ra)$ and so $0 \leq q \leq \frac{m(m-1)}2$. An {\it adapted frame} is a basis $u=(e_1,{\dots},e_{q+m}): \bR^{q+m} \rar \gg$ such that $\gh=\vspan(e_1,{\dots},e_q)$, $\gm=\vspan(e_{q+1},{\dots},e_{q+m})$ and $\la e_{q+i},e_{q+j}\ra=\d_{ij}$. An {\it isomorphism} between two orthogonal transitive Lie algebras $(\gg_i=\gh_i+\gm_i,\la\,,\ra_i)$ is any Lie algebra isomorphism $\f:\gg_1\rar\gg_2$ such that $\f(\gh_1)= \gh_2$, $\f(\gm_1)=\gm_2$ and $\la\,,\ra_1=(\f|_{\gm_1})^*\la\,,\ra_2$. \smallskip

An important class of orthogonal transitive Lie algebras are provided by the Nomizu algebras. Indeed, let $(M,g)$ be a locally homogeneous space, $p \in M$ a distinguished point and $\mathfrak{kill}^g$ the Nomizu algebra of $(M,g,p)$. Consider the Euclidean scalar product on $\mathfrak{kill}^g$ given by $$\la\!\la(v,A),(w,B)\ra\!\ra_g\=g_p(v,w)-\Tr(AB) \,\, ,$$ set $\mathfrak{kill}^g_0\=\{(0,A) \in \mathfrak{kill}^g\}$ and let $\gm$ be the $\la\!\la\,,\ra\!\ra_g$-orthogonal complement of $\mathfrak{kill}^g_0$ in $\mathfrak{kill}^g$. Being $(M,g)$ locally homogeneous, it comes that $\gm \simeq T_pM$ and this allow us to define a scalar product $\la\,,\ra_g$ on $\gm$ induced by the metric $g$ on $M$. Then, $(\mathfrak{kill}^g=\mathfrak{kill}^g_0+\gm,\la\,,\ra_g)$ is an orthogonal transitive Lie algebra. We stress that the Nomizu algebra, modulo isomorphism, does not depend on the choice of the point $p$. Furthermore, two locally homogeneous spaces are locally isometric if and only if their Nomizu algebras are isomorphic.

\begin{rem} Let $(\gg=\gh+\gm,\la\,,\ra)$ be an orthogonal transitive Lie algebra. By \cite[Lemma 3.1]{Sp2}, one can extend $\la\,,\ra$ to an $\ad(\gh)$-invariant Euclidean product $\la\,,\ra'$ on the whole $\gg$ such that $\la\gh,\gm\ra'=0$ and the restriction $\la\,,\ra'|_{\gh \otimes \gh}$ coincides with the Cartan-Killing form of $\so(\gm,\la\,,\ra)$. \label{RemSp} \end{rem}

\subsection{Ambrose-Singer connections and the Singer invariant} \label{Riemtup} \hfill \par

We recall that a smooth Riemannian manifold $(M,g)$ is a locally homogeneous space if and only if it admits an Ambrose-Singer connection, i.e. a metric connection with parallel torsion and curvature \cite[Thm 2.1]{Tr}. In general, Ambrose-Singer connections are far from being unique, but there is always a canonical one, which is characterized as follows. Fix $p \in M$ and for any $k\geq0$ set \beq \gi(k) \= \big\{A \in \so(T_pM,g_p) : \, A \cdot \big((\n^g)^j\Rm(g)_p\big)=0 \,\, , \quad 0 \leq j \leq k \big\} \,\, .\eeq Since $\big(\gi(k)\big)_{k \in \bZ_{\geq0}}$ is a filtration of the finite dimensional Lie algebra $\so(T_pM,g_p)$, there exists a first integer $k_g$ such that $\gi(k_g)=\gi(k_g{+}1)$, and hence $\gi(k)=\gi(k_g)$ for any $k\geq k_g$. It is called the {\it Singer invariant of $(M,g)$}. By \cite{Tr} there exists a unique Ambrose-Singer connection $D^g$ on $(M,g)$ such that $S^g_p \in \gi(k_g)^{\perp}$, where $S^g\=D^g-\n^g$ and $\gi(k_g)^{\perp}$ is the orthogonal complement of $\gi(k_g)$ in $\so(T_pM,g_p)$ with respect to the Cartan-Killing form. 

\begin{rem} Note that for any $v \in T_pM$ the pair $(v,S^g_p(v))$ is always a Killing generator at $p$. Moreover, for globally homogeneous spaces the canonical Ambrose-Singer connection coincides with the {\it canonical connection} according to \cite[Sec X.2]{KN2}. \label{AScan} \end{rem}

By the results in \cite{Si, NT}, it is known that a locally homogeneous space $(M,g)$ with Singer invariant $k_g$ is completely determined by the curvature and its covariant derivatives up to order $k_g+2$ at a single point. For later purposes, for any positive integer $m$ we set \beq \imath(m) \= \max\{k_g : \text{$(M,g)$ loc. hom. with \,$\dim M\leq m$} \} \label{i(m)} \,\, . \eeq Notice that $m \mapsto \imath(m)$ is non-decreasing and $0\leq\imath(m)<\frac32m$ (\cite[p. 165]{Gro}). Moreover, one can directly check that $\imath(1)=\imath(2)=0$ and the main result in \cite{KT} implies that $\imath(3)=\imath(4)=1$. We also mention that from the main theorem in \cite{Meu} we get $\lim_{m\rar+\infty}\imath(m)=+\infty$. 

For any $m,s \in \bN$ with $s \geq \imath(m)+2$, we define $\wt{\cR}^s(m)$ to be the set of all the $(s{+}1)$-tuples $$(R^0,R^1,{\dots},R^s) \in W^s(m) \= \bigoplus_{k=0}^s\Big({\textstyle\bigotimes^k}(\bR^m)^*\otimes\L^2(\bR^m)^*\otimes\so(m)\Big)$$ satisfying the subsequent conditions (R1) and (R2). \\[8pt]
(R1) The following six identities hold: \begin{align*}
&\text{i)}\,\, \la R^0(Y_1{\wedge}Y_2)V_1,V_2\ra_{\st}=\la R^0(V_1{\wedge}V_2)Y_1,Y_2\ra_{\st} \,\, , 	\\
&\text{ii)}\,\, \gS_{{}_{Y_1,Y_2,V_1}}\la R^0(Y_1{\wedge}Y_2)V_1,V_2\ra_{\st}=0 \,\, , \\
&\text{iii)}\,\, \la R^1(X_1|Y_1{\wedge}Y_2)V_1,V_2\ra_{\st}=\la R^1(X_1|V_1{\wedge}V_2)Y_1,Y_2\ra_{\st} \,\, , \\
&\text{iv)}\,\, \gS_{{}_{Y_1,Y_2,V_1}}\la R^1(X_1|Y_1{\wedge}Y_2)V_1,V_2\ra_{\st}=0 \,\, , \\
&\text{v)}\,\, \gS_{{}_{X_1,Y_1,Y_2}}\la R^1(X_1|Y_1{\wedge}Y_2)V_1,V_2\ra_{\st}=0 \,\, , \\
&\text{vi)}\,\, R^{k+2}(X_1,X_2,X_3,{\dots}X_{k+2}|Y_1{\wedge}Y_2)-R^{k+2}(X_2,X_1,X_3,{\dots}X_{k+2}|Y_1{\wedge}Y_2)= \\
&\hskip 91pt =-\big(R^0(X_1{\wedge}X_2) \cdot R^k\big)(X_3,{\dots}X_{k+2}|Y_1{\wedge}Y_2) \quad \text{ for any $0\leq k \leq s-2$ } \,\, ,
\end{align*} where $\gs\go(m)$ acts on the tensor algebra on $\bR^m$ by derivation. \\[8pt]
(R2) For any $1\leq k\leq s$, the maps $$\begin{gathered}
\a^k: \gs\go(m) \rar W^k(m) \,\, , \quad \a^k(A) \= (A\cdot R^0,A\cdot R^1,{\dots},A\cdot R^i) \,\, ,\\
\b^k: \bR^m \rar W^{k-1}(m) \,\, , \quad \b^i(X) \= (X\lrcorner R^1,X\lrcorner R^2,{\dots},X\lrcorner R^i)
\end{gathered}$$ are such that $$\begin{gathered}
\b^k(\bR^m) \subset \a^{k-1}(\gs\go(m)) \quad \text{ for any $\imath(m)+2\leq k \leq s$} \,\, ,\\
\ker(\a^k) = \ker(\a^{k+1}) \quad \text{ for any $\imath(m) \leq k \leq s-1$} \,\, .
 \end{gathered}$$ \vskip 8pt

Notice that $\wt{\cR}^s(m)$ is invariant under the standard left action of $\fO(m)$, and hence  

\begin{definition} Let $m,s \in \bN$ with $s \geq \imath(m)+2$. We call {\it Riemannian $s$-tuples of rank $m$} the elements of the quotient space $\cR^s(m) \= \fO(m) \backslash \wt{\cR}^s(m)$. \end{definition}

This definition is motivated by the following

\begin{theorem}[\cite{Si,NT}] Let $(M^m,g)$ be a locally homogeneous space. Let also $p \in M$ be a point, $u: \bR^m \rar T_pM$ an orthonormal frame and $s \geq \imath(m)+2$ an integer. Then $$\big(u^*\big(\Rm(g)_p\big), u^*\big(\n^g\Rm(g)_p\big), {\dots}, u^*\big((\n^g)^s\Rm(g)_p\big)\big)$$ defines a Riemannian $s$-tuple of rank $m$ which is independent of $p$ and $u$. Conversely, for any Riemannian $s$-tuple $\rho^s \in \cR^s(m)$ of rank $m$, there exists a locally homogeneous Riemannian space $(M^m,g)$, uniquely determined up to a local isometry, such that $\rho^s = \big[u^*\big(\Rm(g)_p\big), u^*\big(\n^g\Rm(g)_p\big), {\dots}, u^*\big((\n^g)^s\Rm(g)_p\big)\big]$ for some $p \in M$ and $u: \bR^n \rar T_pM$ orthonormal frame. \label{NTcurvmod} \end{theorem}

\section{The algebraic convergence and the infinitesimal convergence} \label{algasp} \setcounter{equation} 0

\subsection{The space of locally homogeneous spaces} \hfill \par

For any $m,q \in \bZ$ with $m\geq1$ and $0\leq q \leq \frac{m(m-1)}2$, we indicate with $\eH^{\rm loc}_{q,m}$ the moduli space of transitive orthogonal Lie algebras of rank $(q,m)$ up to isomorphism and we indicate with $\eH_{q,m}$ the subset of moduli space of regular ones. Moreover, we define $$\cV_{q,m} \= \big(\fGL(q) \times \fO(m)\big) \big\backslash \big(\L^2(\bR^{q+m})^*\otimes \bR^{q+m}\big) \,\, ,$$ where we considered a fixed decomposition $\bR^{q+m}=\bR^q\oplus\bR^m$ and the diagonal embedding of $\fGL(q) \times \fO(m)$ into $\fGL(q{+}m)$, which acts on $\L^2(\bR^{q+m})^*\otimes \bR^{q+m}$ on the left by change of basis. By arguing like in \cite{Lau1}, one can show that the map $$\Phi_{q,m}: \eH^{\rm loc}_{q,m} \rar \cV_{q,m} \,\, , \quad (\gg=\gh+\gm,\la\,,\ra) \mapsto \mu \= u^*\big([\cdot,\cdot]_{\gg}\big) \,\, ,$$ where $u: \bR^{m+q} \rar \gg$ is any adapted linear frame for $(\gg=\gh+\gm,\la\,,\ra)$, is well defined, injective and that its image contains all and only the elements $\mu \in \cV_{q,m}$ which verify the following conditions: \begin{itemize}[leftmargin=30pt]
\item[(h1)] $\mu$ satisfies the Jacobi condition, $\mu(\bR^q,\bR^q)\subset\bR^q$ and $\mu(\bR^q,\bR^m)\subset\bR^m$;
\item[(h2)] $\la\mu(Z,X),Y\ra_{\st}=\la X,\mu(Z,Y)\ra_{\st}$ for any $X,Y \in \bR^m$, $Z \in \bR^q$;
\item[(h3)] $\big\{Z\in \bR^q : \mu(Z,\bR^m)=\{0\}\big\}=\{0\}$.
\end{itemize}

\begin{remark} Conditions (h1) and (h2) above are closed, while (h3) is open. Nonetheless, as pointed out in \cite{Lau1}, for any $\tilde{\mu} \in \cV_{q,m} \setminus \eH^{\rm loc}_{q,m}$ satisfying (h1) and (h2), there exists a decomposition $\bR^q= \bR^{q-q'}\oplus\bR^{q'}$ for some $0\leq q' <q$ such that $\bR^{q-q'}=\big\{Z\in \bR^q : \mu(Z,\bR^m)=\{0\}\big\}$ and the restriction \beq (\tilde{\mu})_{|q',m} \= {\rm{pr}}_{\bR^{q'+m}} \circ (\tilde{\mu}|_{\bR^{q'+m}\times\bR^{q'+m}}) \label{restr} \eeq satisfies $(\tilde{\mu})_{|q',m} \in \eH^{\rm loc}_{q',m}$. Here, $\bR^{q'+m}=\bR^{q'}\oplus\bR^m$ and ${\rm{pr}}_{\bR^{q'+m}}: \bR^{q+m} \rar \bR^{q'+m}$ denotes the projection with respect to the direct sum decomposition $\bR^{q+m}=\bR^{q-q'}\oplus\bR^{q'+m}$. \label{remrestr} \end{remark}

From now on, we identify $\eH^{\rm loc}_{q,m}$ with its image through $\Phi_{q,m}$, i.e. we think $\eH^{\rm loc}_{q,m} \subset \cV_{q,m}$. For the sake of clarity, for any $\mu \in \Phi_{q,m}(\eH^{\rm loc}_{q,m}) \simeq \eH^{\rm loc}_{q,m}$ we set $$\gg_{\mu}\=(\bR^{q+m},\mu) \,\, , \quad \gh_{\mu}\=(\bR^{q},\mu|_{\bR^{q}\times\bR^{q}})$$ so that $(\gg_{\mu}=\gh_{\mu}+\bR^m,\la\,,\ra_{\st})$ is the orthogonal transitive Lie algebra uniquely associated to the bracket $\mu$. We also set \beq \eH^{\rm loc}_m \= \bigcup_{q=0}^{\frac{m(m-1)}2}\eH^{\rm loc}_{q,m} \,\, , \quad \eH_m \= \bigcup_{q=0}^{\frac{m(m-1)}2}\eH_{q,m} \,\, . \label{unionq} \eeq

The set $\eH^{\rm loc}_m$ parametrizes the moduli space of the equivalence classes of $m$-dimensional smooth locally homogeneous Riemannian manifolds, up to equivariant local isometries, in the following way.

\begin{theorem}[\cite{Sp1}, Lemma 3.5 and Prop 4.4] For any orthogonal transitive Lie algebra $\mu \in \eH^{\rm loc}_m$, there exist a pointed locally homogeneous space $(M,g,p)$ and an injective homomorphism $\f: \gg_{\mu} \rar \mathfrak{kill}^g$ such that $\f(\gh_{\mu}) \subset \mathfrak{kill}^g_0$ and $(\f(\bR^m),(\f^{-1})^*\la\,,\ra_{\st})=(T_pM,g_p)$, where $\mathfrak{kill}^g$ is the Nomizu algebra of $(M,g,p)$. The space $(M,g,p)$ is uniquely determined up to an equivariant local isometry and it is equivariantly locally isometric to a globally homogeneous space if and only if $\mu$ is regular. \label{lchomass} \end{theorem}

Finally, we mention that the set $\eH_{q,m}$ is dense in $\eH^{\rm loc}_{q,m}$ with respect to the standard topology induced by $\cV_{q,m}$. This fact is known in the literature (see e.g. \cite{Ts}). However, for the convenience of the reader,
we will prove it in Appendix \ref{appendixA}.

\subsection{Curvature of locally homogeneous Riemannian spaces} \label{Curvmu} \hfill \par

Given an element $\mu \in \eH^{\rm loc}_m$, the locally homogeneous Riemannian space uniquely associated to $\mu$ as in Theorem \ref{lchomass} is constructed in the following way. Let $\fG_{\mu}$ be the only simply connected Lie group with $\Lie(\fG_{\mu})=\gg_{\mu}$ and $\fH_{\mu} \subset \fG_{\mu}$ the connected Lie subgroup with $\Lie(\fH_{\mu})=\gh_{\mu}$, which is closed in $\fG_{\mu}$ if and only if $\mu$ is regular. Then one can consider the {\it local quotient} of Lie groups ${\fG}_{\mu}/{\fH}_{\mu}$, which admits a unique suitable real analytic manifold structure (see e.g. \cite[Sec 6]{Ped2}). Moreover, by means of the standard local action of ${\fG}_{\mu}$ on ${\fG}_{\mu}/{\fH}_{\mu}$, one can construct a uniquely determined invariant Riemannian metric $g_{\mu}$ on ${\fG}_{\mu}/{\fH}_{\mu}$ such that $(\bR^m,\la\,,\ra_{\st}) \simeq (T_{e_{\mu}{\fH}_{\mu}}{\fG}_{\mu}/{\fH}_{\mu},g_{\mu}|_{e_{\mu}{\fH}_{\mu}})$. \smallskip

We indicate by $\n^{\mu}$ the Levi-Civita connection of $(\fG_{\mu}/\fH_{\mu}, g_{\mu})$, by $D^{\mu}$ its canonical Ambrose-Singer connection, by $\mathfrak{kill}(\mu)$ its Nomizu algebra, by $\Rm(\mu)$ (resp. $\sec(\mu)$) its Riemannian curvature operator (resp. its sectional curvature) and for any integer $k\geq0$ we set \beq \Rm^k(\mu) : {\textstyle \bigotimes^k}\bR^m \otimes \L^2\bR^m \rar \so(m) \,\, , \quad \Rm^k(\mu)(X_1,{\dots},X_k|Y_1{\wedge}Y_2) \= \big((\n^{\mu})^k_{X_1,{\dots},X_k}\Rm(\mu)\big)(Y_1{\wedge}Y_2) \,\, . \label{defRmk} \eeq Moreover, we consider the orthogonal decomposition \beq \mu = (\mu|_{\gh_{\mu} \wedge \gg_{\mu}})+\mu_{\gh_{\mu}}+\mu_{\bR^m} \,\, , \quad \text{ where } \quad \mu_{\gh_{\mu}}: \bR^m \wedge \bR^m \rar \gh_{\mu} \,\, , \quad \mu_{\bR^m}: \bR^m \wedge \bR^m \rar \bR^m \,\, , \eeq with respect to the the $\ad(\gh_{\mu})$-invariant product $\la\,,\ra'_{\mu}$ on $\gg_{\mu}$ introduced in Remark \ref{RemSp}. \smallskip

Let us consider now the $(1,2)$-tensor $S^{\mu}\= D^{\mu}-\n^{\mu}$, which can be identified with the linear map $S^{\mu}: \bR^m \rar \so(m)$ defined by (see \cite[Thm 3.3, Ch X]{KN2}) \beq -2\la S^{\mu}(X)Y,Z\ra_{\st} = \la\mu_{\bR^m}(X,Y),Z\ra_{\st}+\la\mu_{\bR^m}(Z,X),Y\ra_{\st}+\la\mu_{\bR^m}(Z,Y),X\ra_{\st} \,\, . \label{S1} \eeq Then, by \cite[Thm 2.3, Ch X]{KN2} it follows that \beq \Rm(\mu)(X\wedge Y) = \ad_{\mu}\big(\mu_{\gh_{\mu}}(X,Y)\big)-[S^{\mu}(X),S^{\mu}(Y)]-S^{\mu}(\mu_{\bR^m}(X,Y)) \,\, . \label{Rm0} \eeq Moreover, from \eqref{defRmk} and the fact that $D^{\mu} \Rm^k(\mu)=0$ we get \beq X\lrcorner\Rm^{k+1}(\mu) = -S^{\mu}(X) \cdot \Rm^k(\mu) \quad \text{ for any $k\geq0$ } \,\, . \label{SRm} \eeq Notice that in \eqref{SRm}, for any $A \in \so(m)$ and for any linear map $P: {\otimes}^r \bR^m \rar \so(m)$, the linear map $(A \cdot P): {\otimes}^r \bR^m \rar \so(m)$ is defined by $$(A \cdot P)(V_1,{\dots},V_r) \= [A,P(V_1,{\dots},V_r)] - \sum_{i=1}^rP\big(V_1,{\dots},A \cdot V_i,{\dots},V_r\big) \,\, . $$

Finally, for the sake of notation we denote by \beq \eH^{\rm loc}_{q,m}(1) \= \{ \mu \in \eH^{\rm loc}_{q,m} : |\sec(\mu)|\leq1 \} \,\, , \quad \eH_{q,m}(1)\= \eH^{\rm loc}_{q,m}(1) \cap \eH_{q,m} \label{Hlock} \eeq and we define the spaces $\eH^{\rm loc}_m(1)$ and $\eH_m(1)$ in the same fashion as \eqref{unionq}.

\begin{remark} We stress that for any $\mu \in \eH^{\rm loc}_{q,m}$ there exists $R>0$ such that $R\cdot \mu \in \eH^{\rm loc}_{q,m}(1)$, where the bracket $R\cdot \mu$ is defined by \beq (R\cdot \mu)|_{\gh_{\mu} \wedge \gg_{\mu}} \= \mu|_{\gh_{\mu} \wedge \gg_{\mu}} \,\, , \quad (R\cdot \mu)_{\gh_{\mu}}\=\tfrac1{R^2}\mu_{\gh_{\mu}} \,\, , \quad (R\cdot \mu)_{\bR^m}\=\tfrac1R\mu_{\bR^m} \,\, . \label{scaling} \eeq In fact, the space $(\fG_{R\cdot\mu}/\fH_{R\cdot\mu}, g_{R\cdot\mu})$ is locally equivariantly isometric to $(\fG_{\mu}/\fH_{\mu}, R^2g_{\mu})$ and hence it follows that $\sec(R\cdot\mu)=\tfrac1{R^2}\sec(\mu)$. \label{remscaling} \end{remark}

\subsection{Convergence of locally homogeneous spaces} \label{convlhRs} \hfill \par

We are going to describe two types of convergence in the moduli space $\eH^{\rm loc}_m$. We begin with the algebraic convergence, that is

\begin{definition} A sequence $(\mu^{(n)}) \subset \eH^{\rm loc}_{q,m}$ is said to {\it converge algebraically to $\mu^{(\infty)} \in \eH^{\rm loc}_m$} if one of the two mutually exclusive conditions below is satisfied: \begin{itemize}
\item[i)] $\mu^{(\infty)} \in \eH^{\rm loc}_{q,m}$ and $\mu^{(n)} \rar \mu^{(\infty)}$ in the standard topology induced by $\cV_{q,m}$;
\item[ii)] $\mu^{(\infty)} \in \eH^{\rm loc}_{q',m}$ for some $0\leq q' <q$ and there exists $\tilde{\mu} \in \cV_{q,m} \setminus \eH^{\rm loc}_{q,m}$ such that $\mu^{(n)} \rar \tilde{\mu}^{(\infty)}$ in the standard topology of $\cV_{q,m}$ and $(\tilde{\mu}^{(\infty)})_{|q',m}=\mu^{(\infty)}$ as in Remark \ref{remrestr}.
\end{itemize} \label{algconv} \end{definition}

We introduce also a second notion of convergence. Firstly, let us notice that Theorem \ref{NTcurvmod} and Theorem \ref{lchomass} allow us to consider the map $$\eH^{\rm loc}_m \rar \cR^s(m) \,\, , \quad \mu \mapsto \rho^s(\mu)=\big[\Rm^0(\mu),\Rm^1(\mu),{\dots},\Rm^s(\mu)\big]$$ which associates to any $\mu \in \eH^{\rm loc}_m$ the Riemannian $s$-tuple $\rho^s(\mu)$ of $(\fG_{\mu}/\fH_{\mu},g_{\mu})$, for any $s \geq \imath(m)+2$ (see Subsection \ref{Riemtup}). This map is surjective but not injective. In particular, by Theorem \ref{NTcurvmod} and Theorem \ref{lchomass}, it holds that $\rho^s(\mu_1)=\rho^s(\mu_2)$ for some, and hence for any, $s \geq \imath(m)+2$ if and only if $\mathfrak{kill}(\mu_1)=\mathfrak{kill}(\mu_2)$.

\begin{definition} Let $m,s \in \bN$ with $s \geq \imath(m){+}2$. A sequence $(\mu^{(n)}) \subset \eH^{\rm loc}_m$ is said to {\it converge $s$-infinitesimally to $\mu^{(\infty)} \in \eH^{\rm loc}_m$} if $\rho^s(\mu^{(n)}) \rar \rho^s(\mu^{(\infty)})$ as $n \rar +\infty$ in the standard topology of $\cR^s(m)$. If $(\mu^{(n)})$ converges $s$-infinitesimally to $\mu^{(\infty)}$ for any $s \geq \imath(m)+2$, then $(\mu^{(n)})$ is said to {\it converge infinitesimally to $\mu^{(\infty)}$}. \label{infconv} \end{definition}

Notice that by the previous observation, if a sequence $(\mu^{(n)}) \subset \eH^{\rm loc}_m$ converges $s_1$-infinitesimally to $\mu^{(\infty)}_1$ and $s_2$-infinitesimally to $\mu^{(\infty)}_2$ for some integers $s_2 \geq s_1 \geq \imath(m)+2$, then $\mathfrak{kill}(\mu^{(\infty)}_1)=\mathfrak{kill}(\mu^{(\infty)}_2)$. \smallskip

As we have already pointed out in the Introduction, Lauret introduced in \cite[Sec 6]{Lau1} a notion of infinitesimal convergence for globally homogeneous spaces which is equivalent to the infinitesimal convergence according to Definition \ref{infconv}. For the reader's convenience, we will prove this fact in Appendix \ref{appendixA}. \smallskip

The following proposition puts in relation the algebraic convergence with the infinitesimal convergence (see also \cite[Thm 6.12 (i)]{Lau1} and the errata corrige \cite[Thm 3.9]{Lau4}).

\begin{prop} Let $q,m \in \bZ$ with $m \geq 1$ and $0\leq q \leq \frac{m(m-1)}2$. \begin{itemize}
\item[i)] If $(\mu^{(n)}) \subset \eH^{\rm loc}_{q,m}$ converges algebraically to $\mu^{(\infty)} \in \eH^{\rm loc}_m$, then it converges infinitesimally to $\mu^{(\infty)}$.
\item[ii)] The converse assertion of $(i)$ is not true.
\end{itemize} \label{thmLau} \end{prop}

\begin{proof} Let us assume that $(\mu^{(n)}) \subset \eH^{\rm loc}_{q,m}$ converges algebraically to $\mu \in \eH^{\rm loc}_m$. From \eqref{S1} we immediately get that $S^{\mu^{(n)}} \rar S^{\mu^{(\infty)}}$ in the standard Euclidean topology. Therefore, claim (i) is a direct consequence of \eqref{Rm0} and \eqref{SRm}. On the other hand, claim (ii) follows from \cite[Ex 9.1]{BL}. \end{proof}

Notice that there exist sequences $(\mu^{(n)})$ which admit limits in the infinitesimal convergence, while $|\mu^{(n)}|_{\st}\rar+\infty$. This happens e.g. in considering Ricci flow blow-downs on the universal cover of $\fSL(2,\bR)$ (see \cite{Lo}). This phenomenon, called {\it algebraic collapse}, it has several geometric consequences and it has been investigated in \cite{BL} (see also \cite[Sec 5]{Ped1}).

\subsection{Proof of Theorem \ref{MAIN-A}} \hfill \par

We begin this section by recalling the following

\begin{definition} A {\it geometric model} is a smooth locally homogeneous Riemannian distance ball $(\eB,\hat{g})=(\eB_{\hat{g}}(o,\pi), \hat{g})$ of radius $\pi$ satisfying $|\sec(\hat{g})|\leq 1$ and $\inj_{o}(\eB,\hat{g}_{\mu})=\pi$. \label{geommod} \end{definition}

%From the very definition, we can always assume any geometric model to be of the form $(B^m,\hat{g})$, where $B^m \= B_{\st}(0,\pi) \subset \bR^m$ is the $m$-dimensional Euclidean ball of radius $\pi$, and the standard coordinates of $B^m$ to be normal for $\hat{g}$ at the origin. \smallskip

This special class plays an important role in the theory of convergence of locally homogeneous spaces. Indeed, the main results in \cite{Ped3} can be stated as follows.

\begin{theorem} The following two statements hold true. \begin{itemize}
\item[a)] For each $\mu \in \eH^{\rm loc}_m(1)$, there exists a geometric model $(\eB_{\mu},\hat{g}_{\mu})=(\eB_{\hat{g}_{\mu}}(o_{\mu},\pi), \hat{g}_{\mu})$ which is equivariantly locally isometric to $(\fG_{\mu}/\fH_{\mu}, g_{\mu})$, and it is unique up to global equivariant isometry.
\item[b)] Any sequence of geometric models with bounded geometry up to order $k$ admits a subsequence which converges to a limit geometric model in the pointed $\cC^{k+1,\a}$-topology, for any $0<\a<1$.
\end{itemize} \label{great} \end{theorem}

Claim (a) assures that the geometric models provide a parametrization for $\eH^{\rm loc}_m(1)$ defined in \eqref{Hlock}. In claim (b), a sequence of geometric models $(\eB_{\mu^{(n)}},\hat{g}_{\mu^{(n)}})$ is said to have {\it bounded geometry up to order $k$}, for some integer $k\geq0$, if there exists $C>0$ such that $$\big|\Rm^0(\mu^{(n)})\big|_{\st}+\big|\Rm^1(\mu^{(n)})\big|_{\st}+{\dots}+\big|\Rm^k(\mu^{(n)})\big|_{\st} \leq C \quad \text{ for any $n \in \bN$ .} $$ Moreover, we also recall that a sequence of geometric models $(\eB_{\mu^{(n)}},\hat{g}_{\mu^{(n)}})$ converges in the pointed $\cC^{k+1,\a}$-topology to a geometric model $(\eB_{\mu^{(\infty)}},\hat{g}_{\mu^{(\infty)}})$ if for any $0 < \d < \pi$ there exist an integer $\bar{n}=\bar{n}(\d) \in \bN$ and $\cC^{k+2,\a}$-embeddings $\phi^{(n)}: \eB_{\hat{g}_{\mu^{(\infty)}}}(o_{\mu^{(\infty)}},\pi-\d) \subset \eB_{\mu^{(\infty)}} \rar \eB_{\mu^{(n)}}$ for any $n \geq \bar{n}$ such that $\phi^{(n)}(o_{\mu^{(\infty)}})=o_{\mu^{(n)}}$ and the pulled back metrics $(\phi^{(n)})^*\hat{g}_{\mu^{(n)}}$ converge to $\hat{g}_{\mu^{(\infty)}}$ in $\cC^{k+1,\a}$-topology.

\begin{proof}[Proof of Theorem \ref{MAIN-A}] Let us consider a sequence $(\mu^{(n)}) \subset \eH^{\rm loc}_m(1)$ and assume that $(\mu^{(n)})$ converges $(s{+}1)$-infinitesimally to an element $\mu^{(\infty)} \in \eH^{\rm loc}_m(1)$ for some $s \geq \imath(m){+}2$. By claim (a) in Theorem \ref{great} and the very definition of $(s{+}1)$-infinitesimal convergence, there exists a sequence of geometric models $(\eB_{\mu^{(n)}},\hat{g}_{\mu^{(n)}})$ uniquely associated to $(\mu^{(n)})$ and it has bounded geometry up to order $s+1$. Then, by claim (b) in Theorem \ref{great}, there exists a subsequence of $(\eB_{\mu^{(n)}},\hat{g}_{\mu^{(n)}})$ which converges to a limit geometric model in the pointed $\cC^{s+2,\a}$-topology. Moreover, it holds more: since $(\mu^{(n)})$ converges $(s{+}1)$-infinitesimally to $\mu^{(\infty)}$, by Theorem \ref{NTcurvmod} any convergent subsequence of $(\eB_{\mu^{(n)}},\hat{g}_{\mu^{(n)}})$ in the pointed $\cC^{s+2,\a}$-topology necessarily converges to $(\eB_{\mu^{(\infty)}},\hat{g}_{\mu^{(\infty)}})$. Let us assume then by contradiction that $(\eB_{\mu^{(n)}},\hat{g}_{\mu^{(n)}})$ does not converges to $(\eB_{\mu^{(\infty)}},\hat{g}_{\mu^{(\infty)}})$ in the pointed $\cC^{s+2,\a}$-topology. Then, there exist $0 <\d < \pi$ and a sequence $(n_j) \subset \bN$ such that for any $j_{\zero} \in \bN$ and for any choice of $\cC^{s+3,\a}$-embeddings $\phi^{(n_j)}: \eB_{\hat{g}_{\mu^{(\infty)}}}(o_{\mu^{(\infty)}},\pi-\d) \subset \eB_{\mu^{(\infty)}} \rar \eB_{\mu^{(n_j)}}$ with $\phi^{(n)}(o_{\mu^{(\infty)}})=o_{\mu^{(n_j)}}$ and $j \geq j_{\zero}$, the pulled back metrics $(\phi^{(n_j)})^*\hat{g}_{\mu^{(n_j)}}$ do not converge to $\hat{g}_{\mu^{(\infty)}}$ in $\cC^{s+2,\a}$-topology. On the other hand, the sequence $(\eB_{\mu^{(n_j)}},\hat{g}_{\mu^{(n_j)}})$ has bounded geometry up to order $s{+}1$, and hence it admits a subsequence which converges to a limit geometric model in the pointed $\cC^{s+2,\a}$-topology. By construction, this limit is not isometric to $(\eB_{\mu^{(\infty)}},\hat{g}_{\mu^{(\infty)}})$, which is absurd. \end{proof}

We finally remark that Corollary \ref{MAINCOR-B} follows immediately from Theorem \ref{MAIN-A}.

\section{Collapsing sequences on $\fSU(2)$} \label{AB} \setcounter{equation} 0

\subsection{Riemannian curvature of left invariant metrics on $\fSU(2)$} \label{leftinvSU(2)} \hfill \par

Let us fix a left invariant metric $g$ on the Lie group $\fSU(2)$, which we will constantly identify with the corresponding Euclidean inner product on the Lie algebra $\su(2)$. By the Milnor Theorem, it is known that there exists a $g$-orthogonal basis $(X_0,X_1,X_2)$ for $\su(2)$ such that \beq [X_0,X_1]=-2X_2 \,\, , \quad [X_0,X_2]=+2X_1 \,\, , \quad [X_1,X_2]=-2X_0 \,\, . \label{leftinv} \eeq Up to an automorphism of $\su(2)$, we can suppose that \beq X_0=\left(\!\!\begin{array}{cc} i \!&\! 0 \\ 0 \!&\! -i \end{array}\!\!\right) \,\, , \quad X_1=\left(\!\!\begin{array}{cc} 0 \!&\! -1 \\ 1 \!&\! 0 \end{array}\!\!\right) \,\, , \quad X_2=\left(\!\!\begin{array}{cc} 0 \!&\! i \\ i \!&\! 0 \end{array}\!\!\right) \,\, . \label{stb} \eeq We also set \beq g(X_0,X_0)=\e \,\, , \quad g(X_1,X_1)=\l_1 \,\, , \quad g(X_2,X_2)=\l_2 \,\, . \label{coeffS3} \eeq We aim to study the curvature of $(\fSU(2),g)$. Firstly, by means of \eqref{leftinv} and \eqref{coeffS3} we notice that the bracket $\mu \in \eH_{0,3}$ satisfying $(\fG_{\mu},g_{\mu})=(\fSU(2),g)$ is given by \beq \mu(e_0,e_1)=-2\sqrt{\tfrac{\l_2}{\e\l_1}}\,e_2 \,\, , \quad \mu(e_0,e_2)=+2\sqrt{\tfrac{\l_1}{\e\l_2}}\,e_1 \,\, , \quad \mu(e_1,e_2)=-2\sqrt{\tfrac{\e}{\l_1\l_2}}\,e_0 \,\, . \label{muS3} \eeq Here, we indicated with $(e_0,e_1,e_2)$ the standard basis of $\bR^3$. From now on we will denote by $$\big(E_{ij}{\=}\!e^i\otimes e_j-e^j\otimes e_i \, , \,\, 0 \leq i < j \leq 2\big)$$ the standard basis of $\gs\go(3)$. \smallskip

Let $S^{\mu}: \bR^3 \rar \gs\go(3)$ be the linear operator defined in \eqref{S1}. Then, one can directly check that \beq S^{\mu}(e_0)=+c^{\mu}_0E_{12} \,\, , \quad S^{\mu}(e_1)=-c^{\mu}_1E_{02} \,\, , \quad S^{\mu}(e_2)=+c^{\mu}_2E_{01} \,\, , \label{covder} \eeq where $c^{\mu}_0, c^{\mu}_1, c^{\mu}_2 \in \bR$ are the coefficients defined by \beq c^{\mu}_0 \= {\textstyle\frac{-\e+\l_1+\l_2}{\sqrt{\e\l_1\l_2}}} \,\, , \quad c^{\mu}_1 \= {\textstyle\frac{+\e-\l_1+\l_2}{\sqrt{\e\l_1\l_2}}} \,\, , \quad c^{\mu}_2 \= {\textstyle\frac{+\e+\l_1-\l_2}{\sqrt{\e\l_1\l_2}}} \label{c} \,\, . \eeq In virtue of \eqref{Rm0} and \eqref{covder}, the curvature operator $\Rm^0(\mu): \L^2\bR^3 \rar \gs\go(3)$ is diagonal, i.e. $$\begin{gathered} \Rm^0(\mu)(e_0\wedge e_1)=\sec(\mu)(e_0{\wedge}e_1)E_{01} \,\, , \quad \Rm^0(\mu)(e_1\wedge e_2)=\sec(\mu)(e_1{\wedge}e_2)E_{12} \,\, , \\ \Rm^0(\mu)(e_0\wedge e_2)=\sec(\mu)(e_0{\wedge}e_2)E_{02} \,\, , \end{gathered} $$ and the sectional curvature is given by \begin{align}
&\sec(\mu)(e_0{\wedge}e_1)= -c^{\mu}_0c^{\mu}_1+c^{\mu}_1c^{\mu}_2+c^{\mu}_0c^{\mu}_2 = {\textstyle\frac1{\l_1\l_2}}\big(\e+2(\l_2-\l_1)-\e^{-1}(\l_2-\l_1)(\l_1+3\l_2)\big) \,\, ,\nonumber \\
&\sec(\mu)(e_1{\wedge}e_2)= +c^{\mu}_0c^{\mu}_1-c^{\mu}_1c^{\mu}_2+c^{\mu}_0c^{\mu}_2 = {\textstyle\frac1{\l_1\l_2}}\big({-}3\e+2(\l_1+\l_2)+\e^{-1}(\l_2-\l_1)^2\big) \,\, , \label{sec} \\
&\sec(\mu)(e_0{\wedge}e_2)= +c^{\mu}_0c^{\mu}_1+c^{\mu}_1c^{\mu}_2-c^{\mu}_0c^{\mu}_2 = {\textstyle\frac1{\l_1\l_2}}\big(\e-2(\l_2-\l_1)+\e^{-1}(\l_2-\l_1)(\l_2+3\l_1)\big) \,\, .\nonumber 
\end{align}

\noindent By applying \eqref{SRm}, we obtain the following expressions for the non-zero components of $\Rm^1(\mu)$: \beq \begin{array}{ll}
\Rm^1(\mu)(e_0|e_0\wedge e_1)=2(c^{\mu}_0)^2(c^{\mu}_1{-}c^{\mu}_2)E_{02} \,\, , \quad& \Rm^1(\mu)(e_1|e_1\wedge e_2)=2(c^{\mu}_1)^2(c^{\mu}_0{-}c^{\mu}_2)E_{01} \,\, , \\
\Rm^1(\mu)(e_0|e_0\wedge e_2)=2(c^{\mu}_0)^2(c^{\mu}_1{-}c^{\mu}_2)E_{01} \,\, , \quad& \Rm^1(\mu)(e_2|e_0\wedge e_2)=2(c^{\mu}_2)^2(c^{\mu}_0{-}c^{\mu}_1)E_{12} \,\, , \\
\Rm^1(\mu)(e_1|e_0\wedge e_1)=2(c^{\mu}_1)^2(c^{\mu}_0{-}c^{\mu}_2)E_{12} \,\, , \quad& \Rm^1(\mu)(e_2|e_1\wedge e_2)=2(c^{\mu}_2)^2(c^{\mu}_0{-}c^{\mu}_1)E_{02} \,\, .
\end{array} \label{cov1R} \eeq

\noindent We stress also that from \eqref{SRm}, for any integer $k\geq0$ there exists a constant $C_k >0$, which depends only on $k$, such that \beq \big|\Rm^{k+1}(\mu)(w,v_{i_1},{\dots},v_{i_k})\big|_{\st} \leq C_k \big|S^{\mu}(w)\big|_{\st}\big|\Rm^k(\mu)(v_{i_1},{\dots},v_{i_k})\big|_{\st} \label{boundRm} \eeq for any $w,v_{i_1},{\dots},v_{i_k} \in \bR^3$. Finally, we prove the following

\begin{lemma} Let $k \geq 1$ be an integer and $b(k)\=1+(k \,\,{\rm mod}\, 2)$. Then, it holds that \beq \begin{aligned}
\Rm^k(\mu)(e_0,{\dots},e_0|e_0{\wedge}e_1)&=(-1)^{[\frac{k-1}2]}2^k(c^{\mu}_0)^{k+1}(c^{\mu}_1-c^{\mu}_2)E_{0b(k)} \\
\Rm^k(\mu)(e_0,{\dots},e_0|e_0{\wedge}e_2)&=(-1)^{[\frac{k}2]}2^k(c^{\mu}_0)^{k+1}(c^{\mu}_1-c^{\mu}_2)E_{0b(k+1)} \quad , \\
\Rm^k(\mu)(e_0,{\dots},e_0|e_1{\wedge}e_2)&=0 \label{ind123} \end{aligned} \eeq where $c^{\mu}_0,c^{\mu}_1,c^{\mu}_2$ are defined in \eqref{c}. \end{lemma}
\begin{proof} We proceed by induction on $k \geq 1$. The case $k=1$ follows directly from the computations in \eqref{cov1R}. In order to prove that $k \Rightarrow k+1$, we first notice that \beq [E_{12},E_{0b(k)}]=(-1)^kE_{0b(k+1)} \,\, , \quad (-1)^{[\frac{k}2]}=(-1)^{k-1}(-1)^{[\frac{k-1}2]} \,\, . \label{randomformulas} \eeq So, by \eqref{covder}, \eqref{SRm}, \eqref{randomformulas} and the inductive hypothesis we get \begin{align*}
\Rm^{k+1}(\mu)&(e_0,{\dots},e_0|e_0{\wedge}e_1) = \\
&= -(-1)^{[\frac{k-1}2]}2^k(c^{\mu}_0)^{k+2}(c^{\mu}_1-c^{\mu}_2)[E_{12},E_{0b(k)}]+(-1)^{[\frac{k-1}2]}2^k(c^{\mu}_0)^{k+2}(c^{\mu}_1-c^{\mu}_2)E_{0b(k+1)} \\
&= (-1)^{k-1}(-1)^{[\frac{k-1}2]}2^{k+1}(c^{\mu}_0)^{k+2}(c^{\mu}_1-c^{\mu}_2)E_{0b(k+1)} \\
&= (-1)^{[\frac{k}2]}2^k(c^{\mu}_0)^{k+1}(c^{\mu}_1-c^{\mu}_2)E_{0b(k+1)} \,\, .\end{align*} The second formula is analogous. For the third formula, from \eqref{covder}, \eqref{SRm} and the inductive hypothesis \begin{align*}
\Rm^{k+1}(\mu)&(e_0,{\dots},e_0|e_1{\wedge}e_2) = \\
&=0+c^{\mu}_0\Rm^k(\mu)(e_0,{\dots},e_0|e_2\wedge e_2)-c^{\mu}_0\Rm^k(\mu)(e_0,{\dots},e_0|e_1\wedge e_1) \\
&= 0 \end{align*} and this completes the proof. \end{proof}

\subsection{Almost-Berger sequences} \hfill \par

We consider now a sequence $(\mu^{(n)}) \subset \eH_{0,3}$ of brackets which corresponds to a sequence $(g_{\mu^{(n)}})$ of left-invariant metrics on $\fSU(2)$. By \eqref{leftinv}, we can assume that \beq \begin{gathered} g_{\mu^{(n)}}(X_0,X_0)=\e^{(n)} \,\, , \quad g_{\mu^{(n)}}(X_1,X_1)=\l_1^{(n)} \,\, , \quad g_{\mu^{(n)}}(X_2,X_2)=\l_2^{(n)} \,\, , \\ g_{\mu^{(n)}}(X_0,X_1)=g_{\mu^{(n)}}(X_0,X_2)=g_{\mu^{(n)}}(X_1,X_2)=0 \label{aB} \end{gathered} \eeq where $(X_0,X_1,X_2)$ is the standard basis of $\su(2)$ given in \eqref{stb}. Let us suppose that $$\e^{(n)} \rar 0 \,\, , \quad \l_i^{(n)} \rar \l_i^{(\infty)} \in (0,+\infty) \quad \text{ as $n \rar +\infty$ } . $$ A direct computation based on \eqref{sec} shows that the sectional curvature $\sec(\mu^{(n)})$ is uniformly bounded if and only if \beq \big|\l_1^{(n)}-\l_2^{(n)}\big| \leq C \, \e^{(n)} \quad \text{ for some $C>0$ } \label{l1-l2} \,\, . \eeq Notice that \eqref{l1-l2} is coherent with \cite[Thm 4.3]{Ped1}, which gives necessary conditions for a sequence of invariant metrics to diverge with bounded curvature. \smallskip

Therefore, since we are interested in studying sequences with bounded curvature, we assume without loss of generality that $\lim_{n\to\infty}\l_i^{(n)} =1$, for $i=1,2$. Notice that, if we define \beq \bar{k} \= \sup\Big\{k \in \bZ \, : \big(\e^{(n)}\big)^{-\frac{k}2}\big|\l_1^{(n)}-\l_2^{(n)}\big| \rar 0 \text{ as } n \rar +\infty \Big\} \,\, , \label{regind} \eeq then \eqref{l1-l2} implies that $\bar{k}\geq 1$. We introduce now the following

\begin{definition} An {\it almost-Berger sequence} is any sequence $(\mu^{(n)}) \subset \eH_{0,3}$ which corresponds to a sequence of left-invariant metrics on $\fSU(2)$ as in \eqref{aB} such that \begin{itemize}
\item[i)] $\e^{(n)} \rar 0$ and $\l_i^{(n)}\rar 1$ as $n \rar +\infty$;
\item[ii)] $\sec(\mu^{(n)})$ is uniformly bounded, i.e. \eqref{l1-l2} holds true.
\end{itemize} The positive value ${\rm reg}(\mu^{(n)})\=\bar{k}$ defined in \eqref{regind} is called {\it regularity index of $(\mu^{(n)})$}. \label{aBseq} \end{definition}

This nomenclature is motivated by the following fact. If $\l_1^{(n)} = \l_2^{(n)} = 1$ for any $n \in \bN$, then $g_{\mu^{(n)}}$ comes from the canonical variation of the round metric on $S^3=\fSU(2)$ with respect to the Hopf fibration $S^1 \rar S^3 \rar S^2$ (see \cite[p. 252]{Bes}). The $3$-sphere endowed with such a metric is commonly named {\it Berger sphere}. This construction provides the first non trivial example of collapsing sequence with bounded curvature (see e.g. \cite{CG1,CG2}). \smallskip

From the very definition of regularity index, the following properties hold: \begin{itemize}
\item[$\bcdot$] $\big(\e^{(n)}\big)^{-\frac{k}2}\big|\l_1^{(n)}-\l_2^{(n)}\big| \rar 0$ as $n \rar +\infty$ for any integer $0 \leq k \leq {\rm reg}(\mu^{(n)})$;
\item[$\bcdot$] if ${\rm reg}(\mu^{(n)})=\bar{k}$ is finite, then $\big(\e^{(n)}\big)^{-\frac{\bar{k}+1}2}\big|\l_1^{(n)}-\l_2^{(n)}\big|$ is bounded away from zero.
\end{itemize} Let us stress also that, from \eqref{muS3}, it follows that any almost Berger sequence $(\mu^{(n)})$ verifies $$\mu^{(n)}(e_0,e_1) \rar -\infty \,\, , \quad \mu^{(n)}(e_0,e_2) \rar +\infty \,\, , \quad \mu^{(n)}(e_1,e_2) \rar 0 \quad \text{ as $n\rar +\infty$ } \,\, .$$ This means that they never converge algebraically to an element of $\eH_3$. Actually, the fact that almost-Berger sequences cannot converge algebraically can be derived also from \cite[Prop D]{Ped1}. \smallskip

\subsection{Curvature estimates} \hfill \par

We investigate the behavior of the curvature along an almost-Berger sequence $(\mu^{(n)}) \subset \eH_{0,3}$. Concerning the sectional curvature, from \eqref{sec} we directly get

\begin{prop} Let $(\mu^{(n)})$ be an almost-Berger sequence. Then $$\sec(\mu^{(n)})(e_0{\wedge}e_1) \rar 0 \,\, , \quad \sec(\mu^{(n)})(e_1{\wedge}e_2) \rar 4 \,\, , \quad \sec(\mu^{(n)})(e_0{\wedge}e_2)\rar 0 \quad \text{ as $n \rar +\infty$ }$$ if and only if ${\rm reg}(\mu^{(n)})\geq2$. \label{propsec} \end{prop}

We give now an estimate for the covariant derivatives of the curvature tensor along $(\mu^{(n)})$ in terms of the regularity index. More precisely

\begin{prop} Let $(\mu^{(n)})$ be an almost Berger sequence with regularity index ${\rm reg}(\mu^{(n)})=\bar{k}\geq2$. Then, for any integer $k \geq 1$ there exists a constant $L_k>0$ such that \beq \big|\Rm^k(\mu^{(n)})\big|_{\st} \leq L_k\Big(\big(\e^{(n)}\big)^{\frac12}+\big(\e^{(n)}\big)^{-\frac{k+2}2}\big|\l_1^{(n)}-\l_2^{(n)}\big|\Big) \quad \text{ for any $n \in \bN$ } \, .  \label{crucialestimate} \eeq \end{prop}

\begin{proof} Let us assume, without loss of generality, that $\e^{(n)}<1$ and $\big|\l_1^{(n)}-\l_2^{(n)}\big|<1$ for any $n \in \bN$. Firstly, since $\bar{k}\geq2$ by hypothesis, from \eqref{c} we get $$c^{\mu^{(n)}}_1, c^{\mu^{(n)}}_2 \sim \big(\e^{(n)}\big)^{\frac12}$$ and so from \eqref{covder}, \eqref{cov1R} and \eqref{boundRm} it follows that for any $k\geq1$ there exists $L_{k,0}>0$ such that \beq \big|\Rm^k(\mu^{(n)})(e_{i_1},{\dots},e_{i_k}|e_{j_1}{\wedge}e_{j_2})\big|_{\st}\leq L_{k,0} \big(\e^{(n)}\big)^{\frac{k}2} \label{Rmkno0} \eeq for any $1\leq i_1 , {\dots} , i_k \leq 2$, $0\leq j_1 < j_2 \leq 2$.

Secondly, by \eqref{c} we notice that $$c^{\mu^{(n)}}_0 \sim 2\big(\e^{(n)}\big)^{-\frac12}$$ and so from \eqref{ind123} it follows that for any $k \geq 1$ there exists $L_{k,k}>0$ such that \beq \big|\Rm^k(\mu^{(n)})(e_0,{\dots},e_0|e_{j_1}{\wedge}e_{j_2})\big|_{\st} \leq L_{k,k}\big(\e^{(n)}\big)^{-\frac{k+2}2}\big|\l_1^{(n)}-\l_2^{(n)}\big| \label{Rmkk} \eeq for any $0\leq j_1 < j_2 \leq 2$. Moreover, from \eqref{covder}, \eqref{boundRm} and \eqref{Rmkk}, it follows that for any $k,r \in \bZ$ with $k \geq 2$ and $1\leq r \leq k-1$, there exists $L_{k,r}>0$ such that \beq \big|\Rm^k(\mu^{(n)})(e_{i_1},{\dots},e_{i_{k-r}},\underbrace{e_0,{\dots},e_0}_{r}|e_{j_1}{\wedge}e_{j_2})\big|_{\st}\leq L_{k,r}\big(\e^{(n)}\big)^{-\frac{r+2}2}\big|\l_1^{(n)}-\l_2^{(n)}\big| \label{Rmknor} \eeq for any $1\leq i_1 , {\dots} , i_{k-r} \leq 2$, $0\leq j_1 < j_2 \leq 2$.

Thirdly, a direct computation based on \eqref{SRm} and the last identity in (R1) (see Section \ref{Riemtup}) shows that for any $k,r \in \bZ$ with $k\geq0$ and $0\leq r \leq k$, there exists $N_{k,r}>0$ such that \begin{align}
\Big|\big|\Rm^{k+2}(\mu^{(n)})(e_{\a_1},{\dots},e_{\a_r},e_{\ell_1},e_{\ell_2},e_{\a_{r+1}},{\dots},e_{\a_k}|e_{j_1}&{\wedge}e_{j_2})\big|_{\st} - \nonumber \\
-\big|\Rm^{k+2}(\mu^{(n)})(e_{\a_1},{\dots},e_{\a_r},e_{\ell_2},e_{\ell_1},e_{\a_{r+1}},&{\dots},e_{\a_k}|e_{j_1}{\wedge}e_{j_2})\big|_{\st}\Big| \leq \label{change} \\
&\leq N_{k,r} \sum_{q=0}^r \big|\Rm^q(\mu^{(n)})\big|_{\st} \big|\Rm^{k-q}(\mu^{(n)})\big|_{\st} \nonumber
\end{align} for any $0 \leq \a_1 , {\dots} , \a_k \leq 2$, $0 \leq \ell_1 , \ell_2 \leq 2$, $0\leq j_1 < j_2 \leq 2$.

Finally, we are ready to prove \eqref{crucialestimate} by induction on $k\geq1$. For $k=1$, it follows directly from \eqref{c} and \eqref{cov1R}. Let us fix now $k >1$ and assume that \eqref{crucialestimate} holds for any $1 \leq k' \leq k$. Then

\begin{align*} \big|\Rm^{k+1}(\mu^{(n)})\big|_{\st} &\leq C_1 \sum_{\substack{0\leq\a_1,{\dots},\a_{k+1}\leq2 \\ 0\leq j_1 < j_2 \leq 2}}\big|\Rm^{k+1}(\mu^{(n)})(e_{\a_1},{\dots},e_{\a_{k+1}}|e_{j_1}{\wedge}e_{j_2})\big|_{\st} \\
&\hskip -6pt \overset{\eqref{change}}{\leq} C_2 \Bigg\{ \sum_{r=0}^{k+1}\sum_{\substack{1\leq i_1,{\dots},i_{(k+1)-r}\leq2 \\ 0\leq j_1 < j_2 \leq 2}}\big|\Rm^{k+1}(\mu^{(n)})(e_{i_1},{\dots},e_{i_{(k+1)-r}},\underbrace{e_0,{\dots},e_0}_{r}|e_{j_1}{\wedge}e_{j_2})\big|_{\st}+ \\
&\hskip 213pt +\sum_{q=0}^{k-1} \big|\Rm^q(\mu^{(n)})\big|_{\st} \big|\Rm^{(k-1)-q}(\mu^{(n)})\big|_{\st} \Bigg\} \\
&\hskip -29pt \overset{\eqref{Rmkno0}, \eqref{Rmkk}, \eqref{Rmknor}}{\leq} C_3 \Bigg\{ \big(\e^{(n)}\big)^{\frac{k+1}2}{+}\sum_{i=0}^{k+1}\big(\e^{(n)}\big)^{-\frac{i+2}2}\big|\l_1^{(n)}-\l_2^{(n)}\big|+ \\
&\hskip 213pt +\sum_{q=0}^{k-1} \big|\Rm^q(\mu^{(n)})\big|_{\st} \big|\Rm^{(k-1)-q}(\mu^{(n)})\big|_{\st} \Bigg\} \\
&\leq C_4 \Bigg\{ \big(\e^{(n)}\big)^{\frac12}+\big(\e^{(n)}\big)^{-\frac{k+3}2}\big|\l_1^{(n)}-\l_2^{(n)}\big|+\sum_{q=0}^{k-1} \big|\Rm^q(\mu^{(n)})\big|_{\st} \big|\Rm^{(k-1)-q}(\mu^{(n)})\big|_{\st} \Bigg\}
\end{align*} where $C_i$ are some suitable positive constants. Finally, by Proposition \ref{propsec} and the inductive hypothesis, we obtain $$\sum_{q=0}^{k-1} \big|\Rm^q(\mu^{(n)})\big|_{\st} \big|\Rm^{k-q-1}(\mu^{(n)})\big|_{\st} \leq \tilde{C} \Big( \big(\e^{(n)}\big)^{\frac12}+\big(\e^{(n)}\big)^{-\frac{k+3}2}\big|\l_1^{(n)}-\l_2^{(n)}\big| \Big)$$ for another suitable $\tilde{C}>0$. Therefore, the thesis follows. \end{proof}

\begin{corollary} Let $(\mu^{(n)})$ be an almost Berger sequence with regularity index ${\rm reg}(\mu^{(n)})=\bar{k}$. \begin{itemize}
\item[a)] If $\bar{k}\geq3$, then $\big|\Rm^k(\mu^{(n)})\big|_{\st} \rar 0$ as $n \rar +\infty$ for any $1\leq k \leq \bar{k}-2$.
\item[b)] If $\bar{k}\geq2$ and it is finite, the following conditions hold true: \begin{itemize}
\item[$\bcdot$] $\big|\Rm^{\bar{k}-1}(\mu^{(n)})\big|_{\st}$ does not converge to $0$ as $n \rar +\infty$;
\item[$\bcdot$] $\big|\Rm^{\bar{k}-1}(\mu^{(n)})\big|_{\st}$ is bounded if and only if $\big(\e^{(n)}\big)^{-\frac{\bar{k}+1}2}\big|\l_1^{(n)}-\l_2^{(n)}\big|$ is bounded.
\end{itemize}
\item[c)] If $\bar{k}$ is finite, then $\big|\Rm^k(\mu^{(n)})\big|_{\st} \rar +\infty$ as $n \rar +\infty$ for any integer $k\geq\bar{k}$.
\end{itemize} \label{highcd} \end{corollary}

\begin{proof} Claim (a) follows directly from \eqref{crucialestimate}. Let us assume now that $\bar{k}\geq2$. Then, by \eqref{ind123} we get \beq \big|\Rm^k(\mu^{(n)})(e_0,{\dots},e_0|e_0{\wedge}e_1)\big|_{\st} \sim 2^{2(k+1)}\big(\e^{(n)}\big)^{-\frac{k+2}2}\big|\l_1^{(n)}-\l_2^{(n)}\big| \label{Rm000} \eeq for any $k \geq 1$. Therefore, claim (b) follows from \eqref{crucialestimate} and \eqref{Rm000}, while claim (c) follows from \eqref{Rm000}. Finally, if $\bar{k}=1$, then claim (c) follows from \eqref{cov1R} and \eqref{ind123}. \end{proof}

\section{Proof of the main results} \label{mainsec} \setcounter{equation} 0

Let us consider the $2$-parameter family $$\{\mu_{\star}=\mu_{\star}(\e,\d) : \e,\d \in \bR , \, \e >0 , \, 0\leq\d<1 \} \subset \eH_{0,3}$$ defined by \beq \mu_{\star}(e_0,e_1)=-2\sqrt{\tfrac1{\e}\tfrac{2+\d}{2-\d}}\,e_2 \,\, , \quad \mu_{\star}(e_0,e_2)=+2\sqrt{\tfrac1{\e}\tfrac{2-\d}{2+\d}}\,e_1 \,\, , \quad \mu_{\star}(e_1,e_2)=-2\sqrt{\tfrac{4\e}{4-\d^2}}\,e_0 \,\, .  \label{star} \eeq By means of \eqref{muS3}, each $\mu_{\star}(\e,\d) \in \eH_{0,3}$ correspond to the Lie group $(\fSU(2),g_{\mu_{\star}(\e,\d)})$ with the diagonal left-invariant metric \beq g_{\mu_{\star}(\e,\d)}(X_0,X_0)=\e \,\, , \quad g_{\mu_{\star}(\e,\d)}(X_1,X_1)=1-\tfrac\d2 \,\, , \quad g_{\mu_{\star}(\e,\d)}(X_2,X_2)=1+\tfrac\d2 \,\, , \label{gmustar} \eeq where $(X_0,X_1,X_2)$ is the standard basis of $\su(2)$ given in \eqref{stb}. Let us consider also the Riemannian symmetric space $(\bC P^1 \times\bR, g_{{}_{\rm FS}}{+}dt^2)$, which corresponds to the element $\mu_{\zero} \in \eH_{1,3}$ defined by $$\mu_{\zero}(e_0,e_1)=-2e_2 \,\, , \quad \mu_{\zero}(e_0,e_2)=+2e_1 \,\, , \quad \mu_{\zero}(e_1,e_2)=-2e_0 \,\, , \quad \mu_{\zero}(e_3,\,\cdot\,)=0 \,\, .$$ Clearly it holds that \beq \Rm^0(\mu_{\zero})=\bigg(\!\!{\scalefont{0.6}\begin{array}{ccc} \!4\!\! & & \\ & \!\!0\!\! & \\ & & \!\!0\! \end{array}}\!\!\bigg) \,\, , \quad \Rm^k(\mu_{\zero})=0 \quad \text{ for any $k\geq1$ .} \label{RMmuzero} \eeq

\begin{proof}[Proof of Theorem \ref{MAIN-C}] Let us fix $m=3$, and hence $\imath(3)=1$ (see Section \ref{Riemtup}). Fix an integer $s\geq3$ and consider the sequence $\mu^{(n)} \= \mu_{\star}(\e^{(n)},\d^{(n)})$ defined by $$\e^{(n)}\=\tfrac1{n^2} \,\, , \quad \d^{(n)}\=\tfrac1{n^{s+\frac52}} \,\, .$$ Then we get $$\big(\e^{(n)}\big)^{-\frac{k}2}\big|\l_1^{(n)}-\l_2^{(n)}\big|= \big(\e^{(n)}\big)^{-\frac{k}2}\d^{(n)} =n^{k-s-\frac52} \, , \quad \lim_{n\rar+\infty} n^{k-s-\frac52} =\begin{cases} 0 & \text{ if } 0 \leq k \leq s+2 \\ +\infty & \text{ if } k> s+2 \end{cases}$$ and hence from \eqref{gmustar} it comes that $(\mu^{(n)})$ is an almost Berger sequence with regularity index ${\rm reg}(\mu^{(n)})=s+2$ (see Definition \ref{aBseq}). From Proposition \ref{propsec} and Corollary \ref{highcd} we get $$\begin{gathered} \Rm^0(\mu^{(n)})\rar\bigg(\!\!{\scalefont{0.6}\begin{array}{ccc} \!4\!\! & & \\ & \!\!0\!\! & \\ & & \!\!0\! \end{array}}\!\!\bigg) \,\, , \quad \Rm^k(\mu^{(n)})\rar0 \quad \text{ for any $1\leq k\leq s$ } \,\, , \\ \big|\Rm^{k'}(\mu^{(n)})\big|_{\st}\rar+\infty \quad \text{ for any $k'\geq s+1$ } \,\, . \end{gathered}$$ Therefore the thesis comes directly from \eqref{RMmuzero}. For $m>3$, it is sufficient to consider the Riemannian product $(\fSU(2) \times \bR^{m-3},g_{\mu^{(n)}}{+}g_{\rm flat})$, where $\mu^{(n)}$ is constructed as above choosing $s \geq \imath(m)+2$. \end{proof}

\begin{proof}[Proof of Corollary \ref{MAINCOR-D}] As in the proof of Theorem \ref{MAIN-C}, we can reduce to the case $m=3$. Fix an integer $k \geq \imath(3)+4=5$ and consider the sequences $$\e^{(n)}\=\tfrac1{n^2} \,\, , \quad \d^{(n)}\=\tfrac1{n^{k+1}} \,\, .$$ Then, $\mu^{(n)} \= \mu_{\star}(\e^{(n)},\d^{(n)})$ is an almost-Berger sequence with regularity index ${\rm reg}(\mu^{(n)})=k$ and therefore, by arguing as in the proof of Theorem \ref{MAIN-C}, we get that $(\mu^{(n)})$ converges $(k{-}2)$-infinitesimally to $\mu_{\zero}$. Moreover, since $$\big(\e^{(n)}\big)^{-\frac{k+1}2}\big|\l_1^{(n)}-\l_2^{(n)}\big| = \big(\e^{(n)}\big)^{-\frac{k+1}2}\d^{(n)}=1 \,\, ,$$ by Corollary \ref{highcd} it follows that \beq \tfrac1C<\big|\Rm^{k-1}(\mu^{(n)})\big|_{\st}<C \quad \text{ for some $C>1$ } \,\, , \qquad \big|\Rm^{k'}(\mu^{(n)})\big|_{\st}\to+\infty \quad \text{ for any $k'\geq k$ } \,\, . \label{Rmk-1C} \eeq We can choose $R>0$ big enough in such a way that $|\sec(R\cdot\mu^{(n)})| \leq 1$, where the scaled bracket $R\cdot\mu^{(n)}$ is defined by \eqref{scaling}. Letting $\tilde{\mu}^{(n)} \= R\cdot\mu^{(n)}$ and $\tilde{\mu}_{\zero} \= R\cdot\mu_{\zero}$, by Theorem \ref{great} it follows that we can pass to a subsequence in such a way that $(\eB_{\tilde{\mu}^{(n)}},\hat{g}_{\tilde{\mu}^{(n)}})$ converges to $(\eB_{\tilde{\mu}_{\zero}},\hat{g}_{\tilde{\mu}_{\zero}})$ in the pointed $\cC^{k,\a}$-topology for any $0<\a<1$. Finally, let us assume by contradiction that $(\eB_{\tilde{\mu}^{(n)}},\hat{g}_{\tilde{\mu}^{(n)}})$ admits a convergent subsequence in the pointed $\cC^{k+1}$-topology. By \eqref{RMmuzero}, this would imply that $\big|\Rm^{k-1}(\mu^{(n_j)})\big|_{\st} \rar 0$ as $j \rar +\infty$ for some $(n_j) \subset \bN$, which is impossible by means of \eqref{Rmk-1C}. \end{proof}

\begin{proof}[Proof of Corollary \ref{MAINCOR-E}] Let us first prove that the moduli spaces $\eH_1(1)$, $\eH_2(1)$ are compact in the pointed $\cC^{\infty}$-topology.
For $m=1$, $\eH_1=\eH_{0,1}$ contains only the bracket $\mu=0$ which correspond to the straight line $(\bR=\bR/\{0\},dt^2)$. Therefore in this case the statement is trivially true. For $m=2$, the moduli space $\eH_2(1)$ decomposes as $\eH_2(1)=\eH_{0,2}(1) \cup \eH_{1,2}(1)$. Then, by the well known classification of real $2$-dimensional Lie algebras, we get $$\eH_{0,2}(1)=\{0\} \cup \{R \cdot \mu_{\rm sol} : R \in (0,1] \, \} \,\, ,$$ where the bracket $\mu=0$ correspond to $(\bR^2=\bR^2/\{0\},g_{\rm flat})$, while $\mu_{\rm sol}(e_1,e_2)=e_1$ correspond to the solvmanifold presentation of the hyperbolic space $(\bR H^2=\fG_{\mu_{\rm sol}}/\{0\},g_{\rm hyp})$. On the other hand $$\eH_{1,2}(1)=\{\mu_0\} \cup \{R \cdot \mu_{\epsilon} : R \in (0,1] \, , \,\, \epsilon=\pm1\} \,\, ,$$ where $$\mu_0(e_0,e_1)=+e_2 \,\, , \quad \mu_0(e_0,e_2)=-e_2 \,\, , \quad \mu_0(e_1,e_2)=0$$ corresponds to $(\bR^2=\fSE(2)/\fSO(2),g_{\rm flat})$, while $$\mu_{\epsilon}(e_0,e_1)=+e_2 \,\, , \quad \mu_{\epsilon}(e_0,e_2)=-e_2 \,\, , \quad \mu_{\epsilon}(e_1,e_2)=\epsilon \,e_0 \,\, , \quad \text{ with } \epsilon=\pm1$$ correspond to $(S^2=\fSO(3)/\fSO(2),g_{\rm round})$ and $(\bR H^2=\fSL(2,\bR)/\fSO(2),g_{\rm hyp})$, respectively. Therefore, $\eH_2(1)$ is compact with respect to the algebraic convergence topology and hence the claim follows by Proposition \ref{thmLau} and Corollary \ref{MAINCOR-B}. Finally, to prove that $\eH^{\rm loc}_m(1)$ is not compact in the pointed $\cC^3$-topology for any $m \geq 3$, one can reduce to the case $m=3$ and applying again Corollary \ref{highcd} and Remark \ref{remscaling} in order to construct a sequence $(\mu^{(n)}) \subset \eH_3(1)$ such that $\big|\Rm^1(\mu^{(n)})\big|_{\st}\rar+\infty$ as $n \rar +\infty$. \end{proof} \smallskip

\appendix

\section{} \label{appendixA} \setcounter{equation} 0

\subsection{The subset $\eH_{q,m}$ is dense in $\eH^{\rm loc}_{q,m}$} \hfill \par

The aim of this section is to provide a proof of the following

\begin{prop} The set $\eH_{q,m}$ is dense in $\eH^{\rm loc}_{q,m}$ with respect to the standard topology induced by $\cV_{q,m}$. \label{dense} \end{prop}

In order to do this, we begin with a preliminary lemma.

\begin{lemma} Let $\fT^m=\bZ^m\backslash\bR^m$ be the $n$-torus and let $\fL_v$ be the 1-parameter subgroup of $\fT^m$ generated by a fixed element $v=(v^1,{\dots},v^m) \in \bR^m = \Lie(\fT^m)$. Then $$ \dim \ol{\fL_v} = \dim_{\bQ} \vspan_{\bQ} (v^1,{\dots},v^m) \,\, , $$ where $\vspan_{\bQ} (v^1,{\dots},v^m) \subset \bR$ is the $\bQ$-subspace of $\bR$ generated by the components of $v$. \label{dimLv} \end{lemma}

\begin{proof} We recall that the closed subgroups of a locally compact topological group are characterized as being those which are intersections of kernels of continuous characters (see \cite[Rem 1.20]{Ri}). Moreover, the continuous characters of $\fT^m$ are all and only the applications $$\c_r : \fT^m \rar S^1 \,\, , \quad \c_r(\bZ^mx) \= e^{2\pi \sqrt{-1} \la x,r\ra_{\st}} \quad \text{ with $r \in \bZ^m$ } \, .$$ Since $\fL_v \subset \ker(\c_r)$ if and only if $(\c_r)_*(tv) \in \bZ$ for any $t \in \bR$, i.e. if and only if $\la v,r\ra_{\st}=0$, we get $$\ol{\fL_v} = \bigcap_{\substack{r \in \bZ^n \\ \la v,r\ra_{\st}=0}}\ker(\c_r) \,\, .$$ So by a straightforward computation it comes that $$\dim \ol{\fL_v} = \dim_{\bR}\bigcap_{\substack{r \in \bZ^m \\ \la v,r\ra_{\st}=0}}\ker(\c_r)_* = m-\rank_{\bZ}F \,\, , $$ where we denoted by $F$ the free $\bZ$-module $F\=\{r \in \bZ^m: \la v,r\ra_{\st}=0 \}$. Since the rank of $F$ is $$\rank_{\bZ}F= m-\dim_{\bQ} \vspan_{\bQ} (v^1,{\dots},v^m) \,\, ,$$ the thesis follows. \end{proof}

From Lemma \ref{dimLv}, we get the following corollary. For the definition of Malcev-closure, see Definition \ref{Malcl}.

\begin{corollary} Let $\fT^m=\bZ^m\backslash\bR^m$ be the $m$-torus, ${\rm Gr}_s(\gt)$ the Grassmannian of $s$-planes inside $\gt\=\Lie(\fT^m)$ and ${\rm Gr}_s^{\star}(\gt)$ the subset of those $\ga \in {\rm Gr}_s(\gt)$ such that $\ol{\ga}^{\mathsmaller{\fT^m}}=\ga$. Then ${\rm Gr}_s^{\star}(\gt)$ is dense in ${\rm Gr}_s(\gt)$ for any integer $0\leq s\leq m$. \label{densesubal} \end{corollary}

Finally, we are ready to prove Proposition \ref{dense}.

\begin{proof}[Proof of Theorem \ref{dense}] Fix $\mu \in \eH^{\rm loc}_{q,m} \setminus \eH_{q,m}$. Let $\la\,,\ra'_{\mu}$ be the $\ad(\gh_{\mu})$-invariant Euclidean product on $\gg_{\mu}$ defined in Remark \ref{RemSp} and $\fG_{\mu}$ the simply connected Lie group with $\Lie(\fG_{\mu})=\gg_{\mu}$. Then, the Malcev-closure $\ol{\gh}_{\mu}^{\mathsmaller{\fG_{\mu}}}$ of $\gh_{\mu}$ in $\fG_{\mu}$ turns out to be faithfully represented by its adjoint action as a subalgebra of $\so(\gg_{\mu},\la\,,\ra'_{\mu})$ (see \cite[Sec 3]{Sp2}) and hence it is  reductive. Therefore by \cite[Thm 3, p. 52]{OV} it follows that $$\ol{\gh}_{\mu}^{\mathsmaller{\fG_{\mu}}}=[\gh_{\mu},\gh_{\mu}]+\gt_{\mu} \,\, , \quad \text{ with $\gt_{\mu} \subset \gg_{\mu}$ abelian such that $\ol{\gz(\gh_{\mu})}^{\mathsmaller{\fG_{\mu}}}=\gt_{\mu}$ } .$$ By Corollary \ref{densesubal}, we can pick a sequence of subalgebras $\ga^{(n)} \subset \gt_{\mu}$ which converges to $\gz(\gh_{\mu})$ with respect to the standard Euclidean topology such that $\ol{\ga^{(n)}}^{\mathsmaller{\fG_{\mu}}}=\ga^{(n)}$. Then we define $\gh^{(n)} \= [\gh_{\mu},\gh_{\mu}]+\ga^{(n)}$, $\gm^{(n)}$ as the $\la,\ra'_{\mu}$-orthogonal complement of $\gh^{(n)}$ inside $\gg_{\mu}$ and $\la,\ra^{(n)}\=\la,\ra'|_{\gm^{(n)}\times\gm^{(n)}}$. It is easy to check that $(\gg_{\mu}=\gh^{(n)}+\gm^{(n)},\la,\ra^{(n)})$ is a regular orthogonal transitive Lie algebra, and so it corresponds uniquely to an element $\mu^{(n)} \in \eH_{q,m}$. Finally, by the very construction, we can conclude that $\mu^{(n)} \rar \mu$ as $n \rar +\infty$ in the standard topology induced by $\cV_{q,m}$. \end{proof}

\subsection{The infinitesimal convergence in the sense of Lauret} \hfill \par

In this section, we recall the definition of infinitesimal convergence introduced by Lauret in \cite{Lau1} for sequences of globally homogeneous Riemannian spaces and we prove that it is equivalent to the notion of infinitesimal convergence according to Definition \ref{infconv}. \smallskip

We recall that a sequence $(\mu^{(n)}) \subset \eH_m$ {\it converges infinitesimally to $\mu^{(\infty)} \in \eH_m$ in the sense of Lauret} if there exists a sequence of smooth embeddings $$\phi^{(n)}: \eB_{g_{\mu^{(\infty)}}}\big(e_{\mu^{(\infty)}}{\fH}_{\mu^{(\infty)}},\e^{(n)}\big) \subset {\fG}_{\mu^{(\infty)}}/{\fH}_{\mu^{(\infty)}} \rar {\fG}_{\mu^{(n)}}/{\fH}_{\mu^{(n)}} \,\, , \quad \text{ with $\e^{(n)} \rar 0^+$}$$ such that $\phi^{(n)}(e_{\mu^{(\infty)}}{\fH}_{\mu^{(\infty)}})=e_{\mu^{(n)}}{\fH}_{\mu^{(n)}}$ and \beq \Big|\big((\n^{\mu})^k\big(\phi^{(n)*}g_{\mu^{(n)}}-g_{\mu^{(\infty)}}\big)\big)_{e_{\mu}{\fH}_{\mu}}\Big|_{g_{\mu}} \rar 0 \quad \text{ as $n \rar +\infty$ , \,\, for any integer $k\geq0$ . } \label{infL'} \eeq Fixing a system of local coordinates centered at $e_{\mu}{\fH}_{\mu}$ and letting $g^{(n)}_{ij}$, $g^{(\infty)}_{ij}$ be the components of $\phi^{(n)*}g_{\mu^{(n)}}$ and $g_{\mu^{(\infty)}}$, respectively, it is easy to realize that \eqref{infL'} is equivalent to require that \beq \p^{q}g^{(n)}_{ij}(0) \rar \p^{q}g^{(\infty)}_{ij}(0) \quad \text{ as $n \rar +\infty$ } \label{infL''} \eeq for any multi-index $q$.

\begin{prop} Let $(\mu^{(n)}) \subset \eH_m$ be a sequence, $\mu^{(\infty)} \in \eH_m$. Then, $(\mu^{(n)})$ converges infinitesimally to $\mu^{(\infty)}$ in the sense of Lauret if and only if it converges infinitesimally to $\mu^{(\infty)}$ according to Definition \ref{infconv}. \label{propapp} \end{prop}

Before proceeding with the proof of Proposition \ref{propapp}, for the convenience of the reader we give a detailed proof for the following known fact (see e.g. \cite[Prop E.III.7]{BGM}).

\begin{prop} For any integer $k\geq0$, every partial derivative $\frac{\p^{k+2}g_{ij}}{\p^{q_1}{x^1}{\dots}\p^{q_m}{x^m}}\big|_0$ of order $k+2$ of a real analytic Riemannian metric $g$ in normal coordinates is expressible as a polynomial in the components of $\Rm^0(g)|_0,{\dots},\Rm^k(g)|_0$. \label{normcoord} \end{prop}

\begin{proof} Consider a ball $B\subset \bR^m$ centered at the origin and denote by $g$ a real analytic Riemannian metric on $B$. We also assume that the standard coordinates $(x^1,{\dots},x^m)$ of $\bR^m$ are normal for $g$ at $0$. Then, there exists $\e>0$ sufficiently small such that $$g_{ij}(x)=\d_{ij}+\sum_{k=0}^{\infty}\sum_{|q|=k+2}\frac{\p^qg_{ij}(0)}{q!}x^q \quad \text{ for any $|x|_{\st}<\e$ \,\, ,}$$ where for any multi-index $q=(q_1,{\dots},q_m)$ we set $$q! \= q_1! {\dots} q_m! \,\, , \quad x^q\=(x^1)^{q_1} {\dots} (x^m)^{q_m} \,\, , \quad \p^q\=\frac{\p^{|q|}}{\p^{q_1}{x^1}{\dots}\p^{q_m}{x^m}} \,\, .$$

Fix now $y \in B$ and consider the radial geodesic $\g(t):=ty$ together with the Jacobi vector field $J(t)\=tw^i\tfrac{\p}{\p x^i}$ along $\g(t)$. Set also $f(t)\=g_{\g(t)}(J(t),J(t))$ and, for any tensor field $A=A(t)$, denote by $A^{\{k\}}(t)$ the $k$-th covariant derivative $A^{\{k\}}\=(\n^g_t)^kA$ along $\g(t)$. We recall that the Jacobi equation is $$J^{\{2\}}(t)=R(t)\big(J^{\{0\}}(t)\big) \quad \text{ with } \,\, R(t)(\,\cdot\,)\=-\Rm(g)_{\g(t)}(\dot{\g}(t)\wedge(\,\cdot\,)) \dot{\g}(t) \,\, .$$ By the Leibniz rule, we get \beq \begin{gathered} J^{\{0\}}(0)=0 \,\, , \quad J^{\{1\}}(0)=w  \,\, , \quad J^{\{2\}}(0)=0 \,\, , \\ J^{\{k+2\}}(0) = P_k\big(R^{\{0\}}(0),{\dots},R^{\{k-1\}}(0)\big)w \quad \text{for any integer $k\geq1$} \,\, , \end{gathered} \label{Pk} \eeq where $P_k$ are polynomials in $k$ variables of degree $\deg(P_k)=\lfloor{\frac{k+1}2}\rfloor$ recursively defined by \begin{gather*}
P_1(a^0)\=a^0 \,\, , \quad P_2(a^0,a^1)\=2a^1 \,\, , \\
P_k(a^0,{\dots},a^{k-1})\=ka^{k-1}+\sum_{i=1}^{k-2}\binom{k}{i+2}a^{k-2-i}P_i(a^0,{\dots},a^{i-1}) \,\, .
\end{gather*} Differentiating the function $f$ we get $f(0)=\dot{f}(0)=0$, $\ddot{f}(0)=\la w,w \ra_{\st}$ and for any integer $k\geq1$ \begin{gather} f^{(2k+1)}(0)=2(2k+1)\big\la J^{\{2k\}}(0),w\big\ra_{\st}+\sum_{i=3}^k2\binom{2k+1}{i}\big\la J^{\{2k+1-i\}}(0),J^{\{i\}}(0)\big\ra_{\st} \,\, , \label{f(2k+1)} \\ f^{(2k+2)}(0)=2(2k+2)\big\la J^{\{2k+1\}}(0),w\big\ra_{\st}+\sum_{i=3}^k2\binom{2(k+1)}{i}\big\la J^{\{2k+2-i\}}(0),J^{\{i\}}(0)\big\ra_{\st}+ \nonumber \phantom{aaaaaaaaaaaaaa} \\ \phantom{aaaaaaaaaaaaaaaaaaaaaaaaaaaaaaaaaaaaaaaaaaaaaaaa} +\binom{2(k+1)}{k+1}\big\la J^{\{k+1\}}(0),J^{\{k+1\}}(0)\big\ra_{\st} \,\, . \label{f(2k+2)} \end{gather}

\noindent From \eqref{Pk}, \eqref{f(2k+1)} and \eqref{f(2k+2)} it follow that $$ f^{(k+4)}(0)=\sum_{i,j}\sum_{|q|=k+2}\a^{[k]}_{ijq}\,y^qw^iw^j \quad \text{ for any integer $k\geq0$ }\,\, ,$$ where $\a^{[k]}_{ijq}$ are coefficients which depends polynomially only on the components of $\Rm^0(g)|_0, {\dots}, \Rm^k(g)|_0$. Hence, for $t$ sufficiently small $$ f(t)=\sum_{k=0}^{\infty}\frac{f^{(k)}(0)}{k!}t^k =\d_{ij}w^iw^jt^2+\sum_{k=0}^{\infty}\frac{f^{(k+4)}(0)}{(k+4)!}t^{k+4}=t^2\Bigg(\d_{ij}+\sum_{k=0}^{\infty}\sum_{|q|=k+2}\frac{\a^{[k]}_{ijq}}{(k+4)!}y^qt^{k+2}\Bigg)w^iw^j \,\, . $$ Since $f(t)=t^2g_{ij}(ty)w^iw^j$, we finally get $$g_{ij}(x)=\d_{ij}+\sum_{k=0}^{\infty}\sum_{|q|=k+2}\frac{\tfrac{q!}{(k+4)!}\a^{[k]}_{ijq}}{q!}x^q$$ and the thesis follows. \end{proof}

\begin{proof}[Proof of Proposition \ref{propapp}] Up to scaling, we can assume that $(\mu^{(n)}) \subset \eH_m(1)$ and hence $\mu^{(\infty)} \in \eH_m(1)$. By \eqref{infL''}, it comes that if $(\mu^{(n)})$ converges infinitesimally to $\mu^{(\infty)}$ in the sense of Lauret, then $(\mu^{(n)})$ converges infinitesimally to $\mu^{(\infty)}$ according to Definition \ref{infconv}. On the other hand, if $(\mu^{(n)})$ converges infinitesimally to $\mu^{(\infty)}$ according to Definition \ref{infconv}, then there exists a sequence of matrices $(a^{(n)}) \subset \fO(m)$ such that $a^{(n)} \rar I_m$ and $$a^{(n)} \cdot \Rm^k(\mu^{(n)}) \rar \Rm^k(\mu^{(\infty)}) \quad \text{ as $n \rar +\infty$ } \label{app2}$$ for any integer $k\geq0$, where $a^{(n)}$ acts by change of basis. Therefore Proposition \ref{normcoord} implies that $$(a^{(n)})^{\ell}_i(a^{(n)})^r_j\, \p^q (\hat{g}_{\mu^{(n)}})_{\ell r}(0) \rar \p^q (\hat{g}_{\mu^{(\infty)}})_{ij}(0) \quad \text{ as $n \rar +\infty$ } \label{app1}$$ for any multi-index $q$ and this completes the proof. \end{proof}

\bigskip\bigskip
\font\smallsmc = cmcsc8
\font\smalltt = cmtt8
\font\smallit = cmti8
\hbox{\parindent=0pt\parskip=0pt
\vbox{\baselineskip 9.5 pt \hsize=5truein
\obeylines
{\smallsmc
Dipartimento di Matematica e Informatica ``Ulisse Dini'', Universit$\scalefont{0.55}{\text{\Aac}}$ di Firenze
Viale Morgagni 67/A, 50134 Firenze, ITALY}
\smallskip
{\smallit E-mail adress}\/: {\smalltt francesco.pediconi@unifi.it
}
}
}

\end{document}